\documentclass[preprint,12pt]{elsarticle}
\usepackage{amsmath} \allowdisplaybreaks
\usepackage{amssymb}
\usepackage{amsthm,amsbsy,amsfonts}
\usepackage{mathtools}
\usepackage{graphicx}
\usepackage{verbatim}
\usepackage{xcolor}
\usepackage{enumerate}
\usepackage[shortlabels]{enumitem}
\usepackage{subcaption}
\usepackage{graphicx}
\usepackage{float}
\usepackage[bookmarks=true,hyperindex,pdftex,colorlinks,citecolor=red, linkcolor=blue]{hyperref}

\DeclareMathOperator{\Span}{span}
\DeclareMathOperator{\orb}{orb}
\DeclareMathOperator{\diam}{diam}

\newtheorem{theorem}{Theorem}[section]
\newtheorem{lemma}[theorem]{Lemma}

\newtheorem{corollary}[theorem]{Corollary}

\newtheorem{problem}[theorem]{Problem}
{
\theoremstyle{definition}
\newtheorem{example}[theorem]{Example}
}

\newcommand{\vphi}{\varphi}

\def\RR{\mathbb R}
\def\NN{\mathbb N}

\def\CC{\mathbb C}
\def\N{\mathcal N}
\def\D{\mathcal D}
\def\T{\mathcal T}
\def\R{\mathcal R}
\def\I{\mathcal I}
\def\C{\mathcal C}
\def\H{\mathcal H}

\journal{arXiv}

\begin{document}

\begin{frontmatter}

\title{Self-similar fractals and common hypercyclicity}
\author[aff,grant]{Fernando Costa Jr.}
\ead{fernando@mat.ufpb.br}
\ead[url]{fernando.mat.br}
\fntext[grant]{The author is a Postdoctoral Researcher of the Fonds de la Recherche Scientifique (FNRS) and was also supported by CNPq grant 406457/2023-9.}

\affiliation[aff]{organization={Departamento de Matemática, Universidade Federal da Paraíba - Campus I},
            addressline={Jardim Universitário, s/n, Bairro Castelo Branco},
            city={CEP 58051-900, João Pessoa},
            country={Brazil}}

\begin{abstract}
We obtain a multi-dimensional generalization of the Costakis-Sambarino criterion for common hypercyclic vectors with optimal consequences on a large class of fractals. Applications include families of products of backward shifts parameterized by any Hölder continuous curve in $\mathbb R^d$, for all $d\geq 1$.
\end{abstract}


\end{frontmatter}

\section{Introduction}
\label{sec:intro}

\subsection{Linear dynamics and the common hypercyclicity problem}

Linear dynamics is the study of orbits of linear operators on topological vector spaces. The main notion in this theory is that of hypercyclicity. We say that a continuous endomorphism $T:X\to X$ of a topological vector space $X$ is \emph{hypercyclic} whenever there exists a vector $u$ in $X$ whose orbit $\orb(u;T):=\{T^nu : n\geq 0\}$ under the action of $T$ is dense in $X$. Such a vector is called a \emph{hypercyclic vector} for the operator $T$. It is known that, when $X$ is a complete separable metric space, the set of hypercyclic vectors for an operator $T:X\to X$, which we will denote by $HC(T)$, is either empty or a residual $G_\delta$ subset of $X$ \cite[Theorem 9.20]{GottHedl}. In this case, any countable family $(T_k)_{k}$ of hypercyclic operators acting on $X$ share a vector $u$ which is simultaneously hypercyclic for every one of its members. Indeed, the countable intersection of residual sets $\bigcap_k HC(T_k)$ is non-empty (and also residual). The problem becomes tricky when the family in question is not countable, for a Baire argument cannot be trivially applied. This question, which we will call the \emph{common hypercyclicity} problem, first appeared in the work of G. Godefroy and J. H. Shapiro \cite[Remark 5.5(a)]{GodeShap}. When $\Lambda$ is a $\sigma$-compact subset of some complete metric space, we say that $(T_\lambda)_{\lambda\in \Lambda}$ is a \emph{continuous family} when the map $(\lambda, u)\mapsto T_\lambda u$, from $\Lambda\times X$ into $X$, is continuous in the product topology. A continuous family $(T_\lambda)_{\lambda\in \Lambda}$ of operators is said to be a \emph{common hypercyclic family} when $\bigcap_{\lambda\in\Lambda}HC(T_\lambda)\neq\varnothing$. Any vector $v\in\bigcap_{\lambda\in\Lambda}HC(T_\lambda)$ is called a \emph{common hypercyclic vector} for the family $(T_\lambda)_{\lambda\in \Lambda}$.

The concept of hypercyclicity can be seen as a particular case of the older and more general notion of \emph{universality}. Let $X, Y$ be topological spaces and let $\{T_\xi : \xi \in \N\}$ be a family of mappings from $X$ to $Y$. We say that $u\in X$ is a \emph{universal point} for the family $\{T_\xi : \xi \in \N\}$ whenever its orbit $\orb(x; T_\xi):=\{T_\xi u : \xi \in \N\}$ is dense in $Y$. In this case, we say that $\{T_\xi : \xi \in \N\}$ is a \emph{universal family}. We refer to \cite{Karlsurvey} for a more detailed discussion. In this article, we only consider the particular case where $X$ is a topological vector space, $Y=X$, $\N=\NN$ and $T_n$ is a bounded linear map for all $n\in\NN$. We also add the convention that $T_0=Id$ is the identity map in $X$ for all families $(T_n)_{n\in\NN}$ of maps on $X$. In our context, \emph{points} are called \emph{vectors}, \emph{continuous linear mappings} are called \emph{operators} and hypercyclicity corresponds to universality of the family $T_n=T^n$ of iterates of a single operator. Just like for hypercyclicity, when $\Lambda$ is a $\sigma$-compact subset of some complete metric space, we say that $\{T_{n,\lambda} : n\in\NN, \lambda\in \Lambda\}$ is a \emph{continuous family} when, for each $n\in\NN$, the map $(\lambda, u)\mapsto T_{n,\lambda} u$, from $\Lambda \times X$ into $X$, is continuous  in the product topology.

\subsection{Historical highlights and main results}

The first positive result on the common hypercyclicity problem was obtained in 2003 by E. Abakumov and J. Gordon \cite{AbakGordon} (and independently by A. Peris \cite{PerisCommon}). They answered positively a question raised by H. N. Salas \cite[Section 6(5)]{Salas} on the common hypercyclicity of the classical family $(e^\lambda B)_{\lambda>0}$ of hypercyclic Rolewicz operators (see \cite{Rolewicz}) acting on $\ell_2$, where $B$ is the canonical backward shift. Their proof consists in the explicit construction of the common hypercyclic vector. Further exploring their ideas, in 2004 F. Bayart \cite{bayart2004} provided a common hypercyclic criterion for families of multiples of a single operator, with applications to adjoints of multipliers. Also in \cite{AbakGordon}, the authors mention a very interesting example (attributed to A. Borichev), which can be stated as follows.
\begin{example}[Borichev's example {\cite[Remark 6.3]{BMHowTo}}] \label{Borichev}
    Let $\Lambda=(0,+\infty)\times(0,+\infty)$. For each $\lambda=(x,y)\in \Lambda$, define $T_\lambda=e^xB\times e^y B$ acting on $\ell_2\times \ell_2$. If $K\subset \Lambda$ satisfies $\bigcap_{\lambda\in K}HC(T_\lambda)\neq \varnothing$, then $K$ has Lebesgue measure $0$.
\end{example}
A more general version of this example, along with a more general study of the relation between the Hausdorff dimension of $\Lambda$ and the possibility of $(T_\lambda)_{\lambda\in \Lambda}$ to have a common hypercyclic vector, was further explored in \cite{BCQCommon}. There, Borichev's example was generalized to $d$-dimensional sets of parameters indexing products of weighted backward shifts of a specific (yet general) kind.
\begin{theorem}[{\cite[Corollary 1.3]{BCQCommon}}] \label{dim-small}
Let $X=c_0$ or $\ell_p$, $p\in[1,+\infty)$, $d\in\NN$ and $\alpha\in(0,1]$. Given the family of weights $w(x) = \big(w_1(x), w_2(x),\dots\big)$, $x>0$, that satisfy $w_1(x)\cdots w_n(x)=\exp(n^\alpha)$ for all $n\in\NN$ and $x>0$, if $\Lambda\subset (0,+\infty)^d$ is such that the family $\big(B_{w(\lambda_1)}\times\cdots B_{w(\lambda_d)}\big)_{(\lambda_1,\dots,\lambda_d)\in \Lambda}$ has a common hypercyclic vector, then $\dim_\H(\Lambda) \leq 1/\alpha$.
\end{theorem}
In particular, for $d$-folds of Rolewicz operators $\big(e^{\lambda_1}B\times\cdots \times e^{\lambda_d}B\big)_{(\lambda_1,\dots, \lambda_d)\in \Lambda}$ to be common hypercyclic, it is necessary that $\dim_\H (\Lambda)\leq 1$. We know that this condition is not sufficient, as we can construct a uni-dimensional ``fat'' Cantor dust which does not allow the existence of common hypercyclic vectors for this family (see \cite[Proposition 4.12]{BCQCommon}).

The first common hypercyclicity criterion for general families of operators was obtained in 2004 by G. Costakis and M. Sambarino \cite{CostakisSambarino}, where they proved a very general, yet ``easy-to-use'', criterion for finding common universal vectors. Indeed, the difficult part of obtaining a common hypercyclic vector being to find a good discretization of the parameter set, the Costakis-Sambarino criterion is easy to apply because it includes the construction of the required covering inside its proof (as we shall see below). Let $(X,d)$ be an $F$-space, that is, a topological vector space whose topology is induced by a complete and translation-invariant metric $d:X\times X\to \RR$. We note $\|\cdot\|=d(\cdot, 0)$ the $F$-norm induced by $d$ (see \cite[Definition 2.9]{KarlBook} or \cite{F-spaces} for more details). The Costakis-Sambarino Criterion can be stated as follows.
\begin{theorem}[CS-Criterion {\cite[Theorem 12]{CostakisSambarino}}] \label{cs-1-dim}
Let $\Lambda\subset (0,+\infty)$ be a $\sigma$-compact set and let $\{T_{n,\lambda}:n\in\NN_0, \lambda \in \Lambda\}$ be a continuous family of operators acting on an $F$-space $X$. Assume that there are $\D\subset X$ dense and maps $S_{n,\lambda}:\D\to \D, n\in\NN_0,\lambda\in \Lambda,$ such that $T_{n,\lambda}\circ S_{n,\lambda}=Id$ and, for all $u\in \D$ and $K\subset \Lambda$ compact, the following properties hold true.
    \begin{enumerate}[$(1)$]
        \item \label{CS01} There exist $\kappa\in\NN$ and a summable sequence of positive real numbers $(c_k)_k$ such that
        \begin{enumerate}[$(a)$]
            \item $\|T_{n+k,\lambda}\circ S_{n,\mu}u\|\leq c_k$ for any $n\geq0$, $k\geq \kappa$, $\mu\leq\lambda$, $\lambda,\mu\in K$;
            \item $\|T_{n,\lambda}\circ S_{n+k,\mu}u\|\leq c_k$ for any $n\geq0$, $k\geq \kappa$, $\lambda\leq\mu$, $\lambda,\mu\in K$.
        \end{enumerate}
        \item \label{CS02} Given $\eta>0$, one can find $\tau>0$ such that, for all $n\in\NN$ and all $\lambda,\mu\in K$,
        \[0\leq \mu-\lambda\leq\frac{\tau}{n}\implies \|T_{n,\lambda}\circ S_{n,\mu}u-u\|\leq \eta.\]
    \end{enumerate}
Then the set of common universal vectors for the family $\{T_{n,\lambda}:n\in\NN, \lambda \in I\}$ is a dense $G_\delta$ subset of $X$.
\end{theorem}
For the sake of completeness of the text and clearness of the ensuing discussion, we give a complete proof of this result.
\begin{proof}
We first write $\Lambda=\bigcup_sK_s$, where the union is at most countable and each $K_s=[a_s,b_s]$ is a compact interval. Let $(V_t)_t$ be a countable basis of open sets for $X$ and define $A(s,t)=\{u\in X : \forall \lambda\in K_s, \exists m\in\NN\text{ s.t. }T_{m,\lambda}u\in V_t\}.$ Notice that $\bigcap_{s,t} A(s,t)$ is the set of universal vectors of $\{T_{n,\lambda} : n\in\NN, \lambda\in\Lambda\}$. We shall show that each $A(s,t)$ is open and dense in $X$. Openness comes from the fact that $K_s$ is compact and the family $\{T_{n,\lambda} : n\in\NN, \lambda\in\Lambda\}$ is continuous (we leave the details to the reader). Let $U$ be an arbitrary non-empty open subset of $X$ and let us fix $u_0\in U\cap \D$ and $v_t\in V_t\cap \D$. We take $\eta>0$ small so that $B_{\|\cdot\|}(u_0;\eta)\subset U$ and $B_{\|\cdot\|}(v_t;3\eta)\subset V_t$. Let $\kappa\in\NN$ and $(c_k)_k$ such that property \ref{CS01} is valid for both $u_0$ and $v_t$, and let $\tau>0$ such that property \ref{CS02} holds for $v_t$. We take $N\geq \kappa$ big enough so that $\sum_{k\geq N}c_k<\eta$ and  we define $n_i=iN$, for all $i\in\NN$. Consider the partition of $K_s$ defined inductively by $\lambda_0=a_s$ and $\lambda_i=\lambda_{i-1}+\frac{\tau}{n_i}$ for all $i\in\NN$. Since the series $\sum_i\frac{\tau}{n_i}=\sum_i\frac{\tau}{iN}$ diverges, there is $q\in\NN$ such that $\lambda_{q-1}<b_s\leq \lambda_q$. We reset $\lambda_q:=b_s$ and define $\Lambda_i=[\lambda_{i-1},\lambda_i]$ for $i=1,\dots,q$. Considering $u=u_0+\sum_{i=1}^qS_{n_i,\lambda_i}v_t,$ we show the density of $A(s,t)$ by verifying that $u\in A(s,t)\cap U$. From property \ref{CS01}, we get
\begin{align*}
    \|u-u_0\|&\leq \sum_{i=1}^q\|S_{n_i,\lambda_i}v_t\|
    \leq \sum_{i=1}^qc_{n_i}
    \leq \sum_{k\geq N}c_{k}<\eta,
\end{align*}
which implies that $u\in U$. Now, given $\lambda \in K=\bigcup_{i=1}^q\Lambda_i$, say $\lambda\in \Lambda_{i_0}$ for some $i_0\in\{1,\dots,q\}$, we choose $m=n_{i_0}$ and we apply properties \ref{CS01} and \ref{CS02}, together with $0\leq \lambda_{i_0}-\lambda\leq\frac{\tau}{n_{i_0}}$, in order to get
\begin{align*}
    \|T_{m,\lambda}u-v_t\|
    &\leq 2\sum_{k\geq N}c_k +\eta < 3\eta.
\end{align*}
This implies $T_{m,\lambda}u\in V_t$, that is, $u\in A(s,t)$ and the proof is complete (for a more detailed demonstration, see the \hyperlink{proof-cs}{proof} of Theorem \ref{cs-gamma-dim}).
\end{proof}
It is worth mentioning that there are two equivalent ways of stating condition \ref{CS02} in Theorem \ref{cs-1-dim}. The version presented here, which differs from the original, is an adaptation of \cite[Theorem 7.4]{BayartBook} to the context of universality  (see also \cite[Remark 7.15]{BayartBook} for a detailed discussion). In stark contrast with the original result of E. Abakumov and J. Gordon, the proof of the CS-Criterion does not rely on the construction of the common hypercyclic vector, but rather on the Baire Category Theorem. As we can see in the proof, this is possible to do once the parameter set has been \emph{discretized}, and this discretization satisfy fineness properties that match condition \ref{CS02}. It is interesting to notice that a similar construction cannot be easily adapted to the bi-dimensional family $(T_\lambda)_{\lambda\in(0,+\infty)^2}$ of Example \ref{Borichev}. Indeed, in one dimension we use that $\sum_{i}\frac{\tau}{iN}$ diverges in order to cover an interval by sub-intervals of sides $\frac{\tau}{iN}$, $i\in\NN$, no matter $\tau>0$ and $N\in\NN$, whereas in two dimensions one needs to cover a square by sub-squares of side $\frac{\tau}{iN}$. Thus, each sub-square has area $\frac{\tau^2}{(iN)^2}$ and, unfortunately, the series $\sum_i\frac{\tau^2}{(iN)^2}$ now converges. This problem was already acknowledged by G. Costakis and M. Sambarino in their original paper from 20 years ago, where they suggested that, for families parametrized by $d$-dimensional cubes, one should replace $\frac{\tau}{n}$ in condition \ref{CS02} by something like $\frac{\tau}{n^{1/d}}$ (see \cite[Section 8(2)]{CostakisSambarino}). In \cite[Theorem 1.5]{BCQCommon}, F. Bayart, Q. Menet and the author have obtained a generalization of the CS-Criterion where $\frac{\tau}{n}$ in condition \ref{CS02} is replaced by $\frac{\tau}{n^{\alpha}}$, for any $\alpha \in (0,1/d)$. However, a proof for the optimal statement $\alpha=1/d$, even in the bi-dimensional case, remained open until now. The difficulty is that, even though we know what the correct hypothesis should be, it is not an easy task to actually construct an optimal covering, for there is no ordering in $\RR^d$, when $d\geq2$, that allows us to optimally ``stack'' sub-cubes like we stack sub-intervals in the real line. The aim of this paper is to finally construct this covering.

Of course, the substitution of $\frac{\tau}{n}$ by $\frac{\tau}{n^{1/d}}$ in condition \ref{CS02} of Theorem \ref{cs-1-dim} is not the only correction to be made in a $d$-dimension generalization. In fact, since there is no natural ordering in $\RR^d$ for $d\geq 2$, all inequalities between indexes should be replaced by something else. It is proposed in \cite{BCQCommon} that we request properties \ref{CS01}(a) and \ref{CS01}(b) to hold for all parameters $\lambda$ and $\mu$ such that $\|\lambda-\mu\| \leq Dk^{\alpha}/(n+k)^{\alpha}$ for some $\alpha\in (0,1/d)$ and $D>0$. This assumption is very natural, as is comes from the characterization \cite[Theorem 2.1]{BCQCommon} of families products of weighted backward shifts admitting common hypercyclic vectors. More precisely, this is what is needed when we work with families of weights that behave like the one in Theorem \ref{dim-small} (see \cite[Remark 2.3]{BCQCommon} for more details). We shall prove the following result.
\begin{theorem}[$d$-dimensional CS-Criterion] \label{cs-d-dim}
Let $d\in\NN$, let $\Lambda\subset \RR^d$ be $\sigma$-compact and let $\{T_{n,\lambda} : n\in\NN, \lambda\in\Lambda\}$ be a continuous family of operators acting on an $F$-space $X$. Assume that there are $\D\subset X$ dense, maps $S_{n,\lambda}:\D\to \D$, $n\in\NN$, $\lambda\in\Lambda$, with $T_{n,\lambda}\circ S_{n,\lambda}=Id$, and $D>0$ such that, for all $u\in \D$ and all $K\subset \Lambda$ compact, the following properties hold true.
\begin{enumerate}[$(1)$]
\item There exist $\kappa\in\NN$ and a summable sequence of positive numbers $(c_k)_{k}$ such that, for all $\lambda,\mu\in K$ satisfying $\|\lambda-\mu\|_\infty\leq D\frac{k^{1/d}}{(n+k)^{1/d}}$, with $n,k\in\NN_0$ and $k\geq \kappa$, we have
     \[ \big\| T_{n+k,\lambda}S_{n,\mu} u\big\|\leq c_k \quad\quad\text{and}\quad\quad \big\| T_{n,\lambda}S_{n+k,\mu} u\big\|\leq c_k. \]
\item Given $\eta>0$, one can find $\tau>0$ such that, for all $n\in\NN$ and all $\lambda, \mu\in K,$
\[ \|\lambda-\mu\|\leq \frac{\tau}{n^{1/d}}\implies \left\|T_{n,\lambda} S_{n,\mu} u-u\right\|<\eta. \]
\end{enumerate}
Then the set of common universal vectors for the family $\{T_{n,\lambda}:n\in\NN, \lambda \in \Lambda\}$ is a dense $G_\delta$ subset of $X$. 
\end{theorem}
We shall obtain Theorem \ref{cs-d-dim} as a consequence of a much more general result (Theorem \ref{cs-gamma-dim}), with applications to a large class of fractal curves and shapes. In particular, we will prove the following.
\begin{corollary}\label{corol:intro}
Let $X=c_0$ or $\ell_p$, $p\in[1,+\infty)$, $d\in\NN$ and let $\Lambda\subset (0,+\infty)^d$ be the graph of a $\beta$-Hölder curve, for some $\beta\in(0,1].$ Consider the family of weights $\big(w(x)\big)_{x>0}$ defined by $w_1(x)\cdots w_n(x)=\exp(xn^\alpha)$ for all $n\in\NN$ and $x>0$. Then $\bigcap_{\lambda\in\Lambda}HC\big(B_{w(\lambda_1)}\times \cdots\times B_{w(\lambda_d)}\big)\neq\varnothing$ whenever $\alpha\in(0,\beta]$.
\end{corollary}

More applications and examples will be discussed in Section \ref{sec:app}. In Section \ref{sec:cs-general} we define a slightly different homogeneous box dimension and we state and prove our main result in its general form. In Section \ref{sec:some-fractals} we calculate the dimensions of some well known fractals. In Section \ref{sec:covering}, we explain how we were able to achieve optimally through the construction of a special covering. In Section \ref{sec:prob} we present some open problems in common hypercyclicity.

\subsection{Some notation and terminology}

All sets of parameters that we are going to consider in this article are $\sigma$-compact subsets of the $d$-dimensional euclidean space $\RR^d$, $d\geq 1$, equipped with the maximum norm. We write parameters $\lambda\in\RR^d$ in the form $\lambda=(\lambda(1),\dots, \lambda(d))$. Hence, for any $\lambda,\mu\in\RR^d$, \[\|\lambda-\mu\|:=\max_{j=1,\dots,d}|\lambda(j)-\mu(j)|.\]
For parameters $\lambda\in \RR^2$, we simplify the notation by writing $\lambda=(x,y)$. The \emph{diameter} of any $\Lambda\subset\RR^d$ is defined as \[\diam(\Lambda):=\sup\{\|\lambda-\mu\| : \lambda,\mu\in\Lambda\}.\] The set of positive integers is denoted by $\NN$ and we define $\NN_0:=\NN\cup\{0\}$. We also use the notations $\RR_+:=(0,+\infty)$, $X^d := X\times\overset{d}{\cdots}\times X$ for $d$-folds of a single object $X$ and $\bigoplus_{j=1}^d X_j$ for $d$-folds of objects $X_1,\dots, X_d$. For multi-indexes, given $r,m\in\NN$, we define $I_r=\{1,\dots, r\}$ and we order (when needed) $I_r^m$ with the lexicographical order (sometimes denoted by $\prec$). We use bold letters ${\bf i}\in I_r^m$ for multi-indexes, which are written in the form ${\bf i}=(i_1,\dots,i_m)$.

Let us introduce some terminologies, notations and abbreviations proper to classical fractal geometry of self-similar sets. We refer to \cite{falconer} for a more extensive discussion. Let $d,r\in\NN$, $c_1,\dots, c_r\in(0,1)$ and, for each $i=1,\cdots, r$, let $\vphi_i:\RR^d\to\RR^d$ be a \emph{similarity} with \emph{contraction ratio} (or simply \emph{ratio}) $c_i$, that is, $\vphi_i$ satisfies \[\|\vphi_i(\lambda)-\vphi_i(\mu)\|= c_i\|\lambda-\mu\|, \quad\forall \lambda,\mu\in\RR^d.\] The set $\{\vphi_1, \dots , \vphi_r\}$ is called an \emph{iterated function system} (IFS). We say that it satisfies the \emph{open set condition} (OSC) whenever there exists a non-empty open set $V\subset \RR^d$ such that $\bigcup_{i=1}^d\vphi_i(V)\subset V$ and $\vphi_i(V)\cap \vphi_j(V)=\varnothing$ for all $i\neq j$ in $\{1,\dots, r\}$. The OSC ensures that there is not a significant overlap between the components $\vphi_i(\Lambda)$ of $\Lambda$.

The main feature of an IFS is that there exists a unique compact $\Lambda\subset \RR^d$ such that $\Lambda = \bigcup_{i=1}^r\vphi_i(\Lambda),$
which is called the \emph{attractor} of the IFS $\{\vphi_1, \dots , \vphi_r\}$. In this case, we say that $\Lambda$ is \emph{self-similar} and we define its \emph{similarity dimension} as the number $s\in(0,d]$ such that $\sum_{i=1}^rc_i^s=1$. We know that $s=\dim_\H(\Lambda)$ (the Hausdorff dimension of $\Lambda$) whenever $\{\vphi_1, \dots , \vphi_r\}$ satisfies the OSC. We say that $\Lambda$ is \emph{uniformly contracting} when $c_1=\cdots=c_r=:c$. Notice that, in this case, its similarity dimension is $s=-\frac{\log r}{\log c}$.

Given a self-similar fractal $\Lambda$, attractor of an IFS $\{\vphi_1,\cdots,\vphi_r\}$, one can fix some sufficiently large compact domain $\Lambda_0$ such that $\vphi_i(\Lambda_0)\subset \Lambda_0$, for each $i=1,\dots, r$, and use its IFS to describe $\Lambda$ iteratively. In other words, we can write
\[\Lambda = \bigcap_{m=0}^{+\infty}\bigcup_{{\bf i}\in I_r^m}\vphi_{i_1}\circ\cdots\circ \vphi_{i_m}(\Lambda_0).\]
We usually refer to each $m$ as the \emph{resolution} of the fractal $\Lambda$. Notice that, for each $m$, $\big(\vphi_{i_1}\circ\cdots\circ \vphi_{i_m}(\Lambda_0)\big)_{{\bf i}\in I_r^m}$ is a compact covering of $\Lambda$.

Let us finish this section with a terminology related to \emph{curves}, that is, images $\Lambda=f([0,1])$ of continuous functions $f:[0,1]\subset \RR\to\RR^d$, where $d\geq 2$. We say that $\Lambda$ is a \emph{space-filling curve} when it has non-empty interior. Notice that this definition is not standard, but it is adequate for the purpose of this article. Most known space-filling curves are obtained as the pointwise limit of a sequence of continuous functions $f_m:[0,1]\to \RR^d$, $m\in\NN_0$. When this is the case, we shall call each $f_m([0,1])$ a \emph{pseudo-space-filling curve}. Sometimes the limit curve $\Lambda$ has empty interior, but $\dim_\H>1$. In this case, some authors use the term \emph{fractal-filling curve}, but again this is not standard. In fact, fractals are typically objects with detailed structures at infinitely small scales. In this sense, the graph of the Takagi function is a fractal, although its Hausdorff dimension is 1 (see Problem \ref{prob:takagi}). Whenever a ``fractal curve'' is limit of some sequence $(f_m)_m$, we will use the prefix \emph{pseudo} to indicate the functions $f_m$, $m\in\NN_0$, and we sometimes refer to $m$ as the \emph{resolution} or \emph{order} of the associated pseudo-fractal curve.
\section{A generalized CS-Criterion} \label{sec:cs-general}

In order to state our main result in full generality, we need a specific kind of homogeneous box dimension as defined in \cite[Page 1766]{BCQCommon}. In fact, this new definition changes only for a property that is added to the old one.
\vspace{3mm} 

\noindent {\bf Definition.} Let $d\geq 1$ and $\Lambda\subset \RR^d$ be compact. We say that $\Lambda$ has \emph{homogeneously ordered box dimension} (HBD$^\circ$ for short) \emph{at most} $\gamma\in(0,d]$ if there exist $r\geq 2$, $\rho>0$ and, for all $m\geq 1$, a compact covering $(\Lambda_{\bf i})_{{\bf i}\in I_r^m}$ of $\Lambda$ satisfying the following.
\begin{enumerate}[(i)]
    \item \label{hbd:i} For all ${\bf i}\in I_r^m$, 
    $\diam(\Lambda_{\bf i})\leq \rho(1/r^{1/\gamma})^m.$
    \item \label{hbd:ii} For all ${\bf i}=(i_1,...,i_m,i_{m+1})\in I_r^{m+1}$, $\Lambda_{i_1,...,i_{m+1}}\subset \Lambda_{i_1,...,i_{m}}.$
    \item \label{hbd:iii} For all ${\bf i}=(i_1,...,i_m)\in I_r^m$ and all $j\in\{2,...,r\}$, $\Lambda_{{\bf i},j-1,r}\cap \Lambda_{{\bf i}, j,1}\neq \varnothing.$
\end{enumerate}
The \emph{homogeneously ordered box dimension of $\Lambda$} is the smallest number $\gamma$ such that $\Lambda$ has HBD$^\circ$ at most $\gamma$. Of course, whenever a compact subset $\Lambda\subset \RR^d$ has HBD$^\circ$ at most $\gamma$, then it also has HBD at most $\gamma$. In many interesting cases (as we shall see in Section \ref{sec:some-fractals}), the properties of the HBD discussed in \cite[Section 4.2]{BCQCommon} remain true for the HBD$^\circ$.

\begin{figure}[H]
    \centering
    \includegraphics[width=0.75\textwidth]{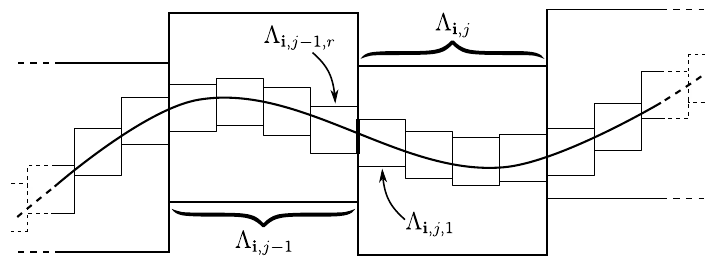}
    \caption{Connectedness property of the homogeneously ordered box dimension}
    \label{fig:ord}
\end{figure}

The addition of property \ref{hbd:iii} in the definition of HBD$^\circ$ imposes, at the same time, path connectivity and ordering on the family of coverings. Indeed, if ${\bf j}$ comes right after ${\bf i}$ in the lexicographical order, then \ref{hbd:iii} requires that the last subdivision part of $\Lambda_{\bf i}$ intersects the first subdivision part of $\Lambda_{\bf j}$ (see Figure \ref{fig:ord} for a graphical representation).

Later on we will discuss how the possibility of calculating this dimension allow us to define a continuous fractal-filling curve (see Section \ref{sec:prob}). With this concept of dimension at hand, we can state our main result.

\begin{theorem}[Generalized CS-Criterion] \label{cs-gamma-dim}
Let $d\in\NN$, $\gamma\in\RR^+$, $\Lambda\subset \RR^d$ compact with HBD$^\circ$ at most $\gamma$ and let $\{T_{n,\lambda} : n\in\NN, \lambda\in\Lambda\}$ be a continuous  family of operators acting on an $F$-space $X$. Assume that there are $\D\subset X$ dense, maps $S_{n,\lambda}:\D\to \D$, $n\in\NN$, $\lambda\in\Lambda$, such that $T_{n,\lambda}\circ S_{n,\lambda}=Id$ and $D>0$ such that, for all $u\in \D$, the following properties hold true.
\begin{enumerate}[start=1,label={\upshape(CS\arabic*)},wide = 0pt, leftmargin = 3em]
\item \label{CS1} There exist $\kappa\in\NN$ and a summable sequence of positive numbers $(c_k)_{k}$ such that, for all $\lambda,\mu\in K$ satisfying $\|\lambda-\mu\|\leq D\frac{k^{1/\gamma}}{(n+k)^{1/\gamma}}$, with $n\geq 0$ and $k\geq \kappa$, we have
     \[ \big\| T_{n+k,\lambda}\circ S_{n,\mu} u\big\|\leq c_k \quad\quad\text{and}\quad\quad \big\| T_{n,\lambda}\circ S_{n+k,\mu} u\big\|\leq c_k. \]
\item \label{CS2} Given $\eta>0$, one can find $\tau>0$ such that, for all $n\in\NN$ and all $\lambda,\mu\in K$,
\[ \|\lambda-\mu\|\leq \frac{\tau}{n^{1/\gamma}}\implies \left\|T_{n,\lambda}\circ S_{n,\mu} u-u\right\|\leq\eta. \]
\end{enumerate}
Then the set of common universal vectors for the family $\{T_{n,\lambda}:n\in\NN, \lambda \in \Lambda\}$ is a dense $G_\delta$ subset of $X$.
\end{theorem}

The proof of Theorem \ref{cs-gamma-dim} requires the construction of a special covering for the set of parameters satisfying good proximity properties. We shall need the following key lemma.

\begin{lemma}\label{key:lemma} Let $d\in\NN$, $\gamma\in \RR_+$ and $\Gamma\subset \RR^d$ be compact with with HBD$^\circ$ at most $\gamma$. Let $D\geq \rho r^{3/\gamma}$, where $r\geq 2,\rho>0$ come from the definition of HBD$^\circ$ applied to $\Gamma$. Then, for all $\tau>0$ and all $N\in\NN$, there exist $q\in\NN$ and a tagged covering $(\lambda_k,\Gamma_k)_{k=1,...,q}$ of $\Gamma$ having the form
\begin{equation}\label{fraise}
    \Gamma_k=\prod_{i=1}^d\bigg[\lambda_k(i), \lambda_k(i)+\frac{\tau}{(kN)^{1/\gamma}}\bigg],
\end{equation}
for some $\lambda_k=(\lambda_k(1), \dots, \lambda_k(d)), k=1,..., q$, and satisfying
    \begin{equation}\label{cerise}
    \|\lambda-\mu\|\leq D\Big(\frac{l-j}{l}\Big)^{1/\gamma},
    \end{equation}
for all $1\leq j < l \leq q$ and all $\lambda\in\Gamma_j, \mu\in\Gamma_l$.
\end{lemma}

With this covering result at hand, the proof of our main theorem goes as follows.

\begin{proof}[Proof of Theorem \ref{cs-gamma-dim}.] \hypertarget{proof-cs}{}
For simplicity, we make the proof for $d=2$ (the adaptation to the general case is straightforward), so we can write tags in the form $\lambda_j=(x_j, y_j)$. Since all covering parts in the family of coverings from the definition satisfy the same covering properties themselves (for the same $r$ and smaller $\rho$'s), we can decompose $\Lambda=\bigcup_s K_s$ in a (finite) union of small compact subsets $K_s$ all having HBD$^\circ$ at most $\gamma$ for the same $r\geq 2$ as $\Lambda$ and for $\rho>0$ small so that $D\geq \rho r^{3/\gamma}$. Let $(V_t)_t$ be a countable basis of open sets for $X$ and define \[A(s,t)=\{u\in X : \forall \lambda\in K_s, \exists m\in\NN\text{ s.t. }T_{m,\lambda}u\in V_t\}.\] Notice that $\bigcap_{s,t} A(s,t)$ is the set of universal vectors of $\{T_{n,\lambda} : n\in\NN, \lambda\in\Lambda\}$. Also notice that each $A(s,t)$ is open. We shall prove that, in addition, $A(s,t)$ is dense. Let $U$ be an arbitrary non-empty open subset of $X$ and let us fix $u_0\in U\cap \D$ and $v_t\in V_t\cap \D$. We take $\eta>0$ small so that $B_{\|\cdot\|}(u_0;\eta)\subset U$ and $B_{\|\cdot\|}(v_t;3\eta)\subset V_t$. Let $\kappa\in\NN$ and $(c_k)_k$ such that property \ref{CS1} is valid for both $u_0$ and $v_t$, and let $\tau>0$ such that property \ref{CS2} holds for $v_t$. We take $N\geq \kappa$ big enough so that $\sum_{k\geq N}c_k<\eta$ and we define $n_i=iN$, for all $i\in\NN$. We apply Lemma \ref{key:lemma} for these values of $\tau$, $N$ and $D$ in order to obtain $q\in\NN$ and a tagged covering $(\lambda_i,\Gamma_i)_{i=1,...,q}$ of $K_s$ satisfying \eqref{fraise} and \eqref{cerise}. In particular, 
\[
\Gamma_i =\biggl[x_i, x_i+\frac{\tau}{n_i^{1/\gamma}}\biggr]\times\biggl[y_i, y_i+\frac{\tau}{n_i^{1/\gamma}}\biggr], \quad 
i=1,\dots,q.
\]
Hence, for every $i=1,\dots,q$, 
\begin{equation}\label{eq:prox:part}
    \lambda\in\Gamma_i\implies \|\lambda-\lambda_i\|\leq \frac{\tau}{n_i^{1/\gamma}}.
\end{equation}
Moreover, from property \eqref{cerise} and the definition of $(n_i)_i$ we get that, for all $1\leq j<l\leq q$ and all $\lambda\in \Gamma_j$, $\mu\in\Gamma_l$,
\begin{equation}\label{eq:prox:tag}
    \|\lambda-\mu\|\leq D\Big(\frac{l-j}{l}\Big)^{1/\gamma} = D\Big(\frac{n_l-n_j}{n_l}\Big)^{1/\gamma}.
\end{equation}
Considering $u=u_0+\sum_{i=1}^qS_{n_i,\lambda_i}v_t,$ we show the density of $A(s,t)$ by verifying that $u\in A(s,t)\cap U$. For each $i\in\NN$, we use \eqref{eq:prox:tag} in order to apply \ref{CS1} with $k=n_i$ and $n=0$ (notice that $D>\diam(K_s)=\sup_{\lambda,\mu \in K_s}\|\lambda-\mu\|$)
and get
\begin{equation}\label{proof:prop:0}
    \|T_{n_i,\lambda}u_0\|\leq c_{n_i}=c_{iN}\quad\text{and}\quad \|S_{n_i,\lambda}v_t\|\leq c_{n_i}=c_{iN}, \quad\forall\lambda\in K_s.
\end{equation}
We then have
\[\|u-u_0\|\leq \sum_{i=1}^q\|S_{n_i,\lambda_i}v_t\|\leq \sum_{i=1}^q c_{iN}\leq \sum_{k\geq N}^{+\infty} c_k < \eta \implies u\in U.\]
With the purpose of showing that $u\in A(s,t)$, let $\lambda\in K_s$. Then there exists $i\in\{1,\dots,q\}$ such that $\lambda\in\Gamma_i$. We then choose $m=n_i$. For each $j=i+1,\dots,q$, we use \eqref{eq:prox:tag} and apply \ref{CS1} with $k=n_j-n_i$ and $n=n_i$ in order to get 
\begin{equation}\label{eq:i<j}
    \|T_{n_i,\lambda}\circ S_{n_j,\lambda_j}v_t\|\leq c_{n_j-n_i}=c_{(j-i)N}.
\end{equation}
Similarly, for $j=1,\dots, i-1$ we apply \ref{CS1} with $k=n_i-n_j$ and $n=n_j$ and we find
\begin{equation}\label{eq:i>j}
    \|T_{n_i,\lambda}\circ S_{n_j,\lambda_j}v_t\|\leq c_{n_i-n_j} = c_{(i-j)N}. 
\end{equation}
Finally, since $\lambda\in\Gamma_i$, from \eqref{eq:prox:part} and \ref{CS2} we get
\begin{equation}\label{eq:i=j}
    \|T_{n_i,\lambda}\circ S_{n_i,\lambda_i}v_t-v_t\|<\eta.
\end{equation}
Therefore, we use \eqref{proof:prop:0}, \eqref{eq:i<j}, \eqref{eq:i>j} and \eqref{eq:i=j} and we get
\begin{align*}
    \|T_{m,\lambda}u-v_t\| 
    &= \|T_{n_{i},\lambda}u_0\|+\sum_{j\neq i}\|T_{n_{i},\lambda}\circ S_{n_j,\lambda_j}v_t\|+\|T_{n_{i},\lambda}\circ S_{n_{i}}v_t-v_t\| \\
    &\leq c_{iN}+ \sum_{j=1}^{i-1}c_{(i-j)N}+\sum_{j=i+1}^qc_{(j-i)N}+\eta \\
    &\leq 2\sum_{j\geq N}c_j +\eta < 3\eta,
\end{align*}
what gives $T_{m,\lambda}u\in V_t$, that is, $u\in A(s,t)$. This concludes the proof.
\end{proof}

Just as in \cite[Theorem 3.3]{BMHowTo}, we can substitute \ref{CS2} in Theorem \ref{cs-gamma-dim} by condition \ref{CS2p} below, which is the form where it is generally applied. 
\begin{enumerate}[start=2,label={\upshape(CS\arabic*')},wide = 0pt, leftmargin = 3em]
\item There is $C>0$ such that, for all $n\in\NN$ and all $\lambda, \mu\in K$, \label{CS2p} \[ \left\|T_{n,\lambda}\circ S_{n,\mu} u-u\right\|\leq Cn^{1/\gamma}\|\lambda-\mu\|.\]
\end{enumerate}
The implication \ref{CS2p}$\implies$\ref{CS2} is immediate.

The proof of Lemma \ref{key:lemma}, that is, the construction of this special covering with the desired properties, is postponed to Section \ref{sec:lemma:proof}.
\section{Some fractals and their homogeneously ordered box dimensions} \label{sec:some-fractals}

In this section we calculate explicitly the HBD$^\circ$ of some popular fractals. This will be important for us to obtain (sometimes optimal) applications, as we are going to discuss in Section \ref{sec:app}. We begin with the Sierpi\'nski gasket because it is an iconic self-similar fractal that has only 3 similarities, what makes the discussion simpler and shorter.

\subsection{Sierpi\'nski gasket}

First defined in 1915 by the Polish mathematician Wacław Sierpiński \cite{sierp1915} as the limit of curves (the so-called Sierpi\'ski arrowhead curve), the Sierpi\'nski gasket (SG for short) is a fractal obtained by the successive removal of inverted equilateral triangles inside equilateral triangles. The SG is a self-similar fractal with $r=3$ similarities and uniform contraction ratio $c=1/2$. Hence, it has similarity dimension $\gamma=-\frac{\log r}{\log c}=\frac{\log3}{\log2}$, which coincides with its Hausdorff dimension since the SG satisfies the OSC. One can easily define a refining family of coverings and find that $\gamma$ is also its homogeneous box dimension. The fact that its homogeneously ordered box dimension also equals $\gamma$ can be seen by defining a conveniently ordered IFS of which the SG is the attractor. Let us do just that.

\begin{figure}
    \centering
    \includegraphics[width=\textwidth]{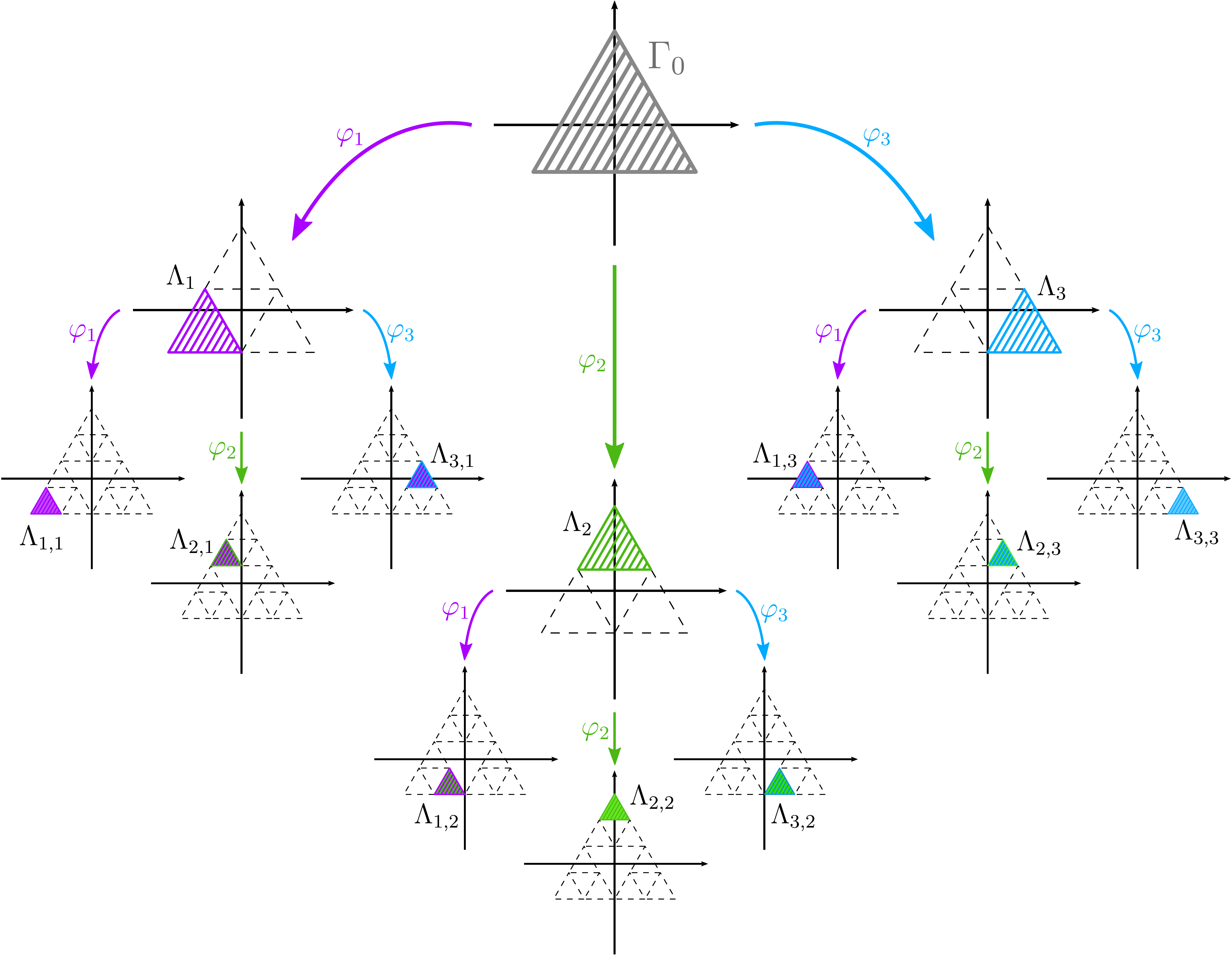}
    \caption{IFS generating the homogeneously ordered family of coverings for the Sierpi\'nski gasket}
    \label{fig:ifs-sg}
\end{figure}

Let $\Gamma\subset \RR^2=\CC$ be the SG of side $\ell$ and barycentre at the origin of the complex plane. We define, for each $z,w\in \CC$, $c\in(0,1)$ and $\theta\in(-\pi,\pi]$,
\begin{alignat*}{2}
    \I_v(z)&=\Bar{z} &&  \quad \text{vertical reflection},\\
    \I_h(z)&=-\Bar{z} &&  \quad \text{horizontal reflection},\\
    \T_{w}(z)&=z+w &&  \quad \text{translation by }w,\\
    \C_c(z)&=cz &&  \quad \text{contraction with ratio }c,\\
    \R_{\theta}(z)&=ze^{i\theta} &&  \quad \text{rotation by }\theta.
\end{alignat*}
The fundamental similarities $\vphi_k:\CC\to\CC, k=1,2,3,$ defining $\Gamma$ are
\begin{gather*}
    \vphi_1=\T_{w_1}\circ \C_{\frac{1}{2}}\circ \R_{\frac{\pi}{3}}\circ \I_{v},\quad
    \vphi_2=\T_{w_2}\circ\C_{\frac{1}{2}}\\
    \text{and}\quad\vphi_3=\T_{w_3}\circ \C_{\frac{1}{2}}\circ \R_{\frac{-\pi}{3}}\circ \I_{v},
\end{gather*}
where $w_1=\big(\frac{-1}{4}-i\frac{\sqrt{3}}{12}\big)\ell$, $w_2=i\frac{\sqrt{3}}{6}\ell$ and $w_3=\I_w(w_1)$ are the fundamental translations (barycenters of the three triangles in the first iteration). Explicitly, we have, for all $z\in \CC$,
\begin{gather*}
    \vphi_1(z)=\frac{1}{2}\Bar{z}e^{i\frac{\pi}{3}}-\frac{\ell}{4}-i\frac{\sqrt{3}}{12}\ell, \quad 
    \vphi_2(z)=\frac{1}{2}z+i\frac{\sqrt{3}}{6}\ell  \\ 
    \text{and}\quad\vphi_3(z)= \frac{1}{2}\Bar{z}e^{-i\frac{\pi}{3}}+\frac{\ell}{4}-i\frac{\sqrt{3}}{2}\ell.
\end{gather*}
Then $\Gamma$ is clearly the attractor of the IFS $\{\vphi_1,\vphi_2,\vphi_3\}.$ By taking $\Gamma_0$ as the equilateral triangle of side $\ell$ and barycentre at the origin, we have that the family of coverings $\big((\Lambda_{\bf i})_{{\bf i}\in I_3^m}\big)_{m\in\NN}$ defined by 
\[
\Lambda_{i_1,\dots i_m} = (\vphi_{i_1}\circ \cdots \circ \vphi_{i_m})(\Gamma_0), \quad \text{for all }m\in \NN,
\]
is homogeneously ordered (see Figure \ref{fig:ifs-sg} for the first 2 iterations of the construction). Therefore, one can see that $\Gamma$ has homogeneously ordered box dimension at most $\frac{\log3}{\log2}$. Indeed, it is plain that $(\Lambda_{\bf i})_{{\bf i}\in I_3^m}$ is a covering of $\Gamma$ for all $m\in \NN$. Furthermore, condition \ref{hbd:i} comes from the fact that $\vphi_1,\vphi_2$ and $\vphi_3$ are all contractions of ratio $c:=1/2$ (notice that $1/r^{1/\gamma}=c$), whereas conditions \ref{hbd:ii} and $\ref{hbd:iii}$ follows from the construction of this IFS.

\begin{figure}[H]
    \centering
    \includegraphics[width=0.80\textwidth]{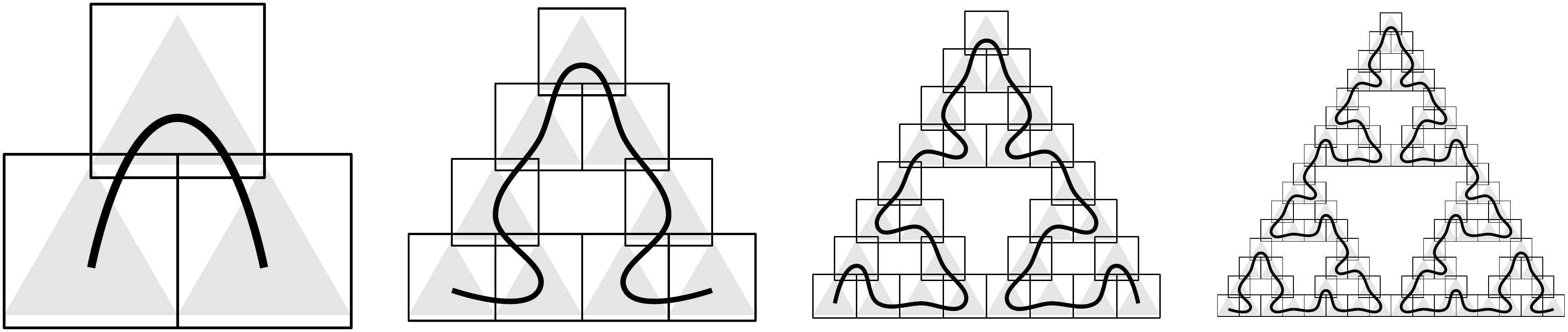}
    \caption{Homogeneous ordering of the family of coverings for the Sierpiński gasket}
    \label{fig:sierp-ord}
\end{figure}

Notice that the orderings of the coverings that we have obtained follow exactly the trace of the curves as in Figure \ref{fig:sierp-ord}. Thus, although less formal, we could simply say that the ordering of the coverings follows the pseudo-arrowhead curve of respective resolution.

\subsection{Pseudo-Hilbert curves filling a square}

Following the common sense on the word ``fractal'', squares would probably not be included in this category. However, squares can arguably be regarded as fractals once it is obtained through a fractal construction. Even the whole plane (or space) can be considered a fractal in the form of a space-filling curve. The \emph{Peano curve} is the very first manifestation of these kinds of curves. It was defined in 1890 by the Italian mathematician Giuseppe Peano \cite{peano1890} and is a self-similar fractal with 9 similarities. In 1891, the German mathematician David Hilbert \cite{hilbert1891} proposed an alternative and much simpler construction, the so called \emph{Hilbert curve}, which has only 4 similarities. In Figure \ref{fig:square-ordering} we see the first four iterations in the construction of the Hilbert space-filling curve (we refer to \cite[Chapter 2]{hans} for a more detailed discussion).

\begin{figure}[H]
    \centering
    \includegraphics[width=0.80\textwidth]{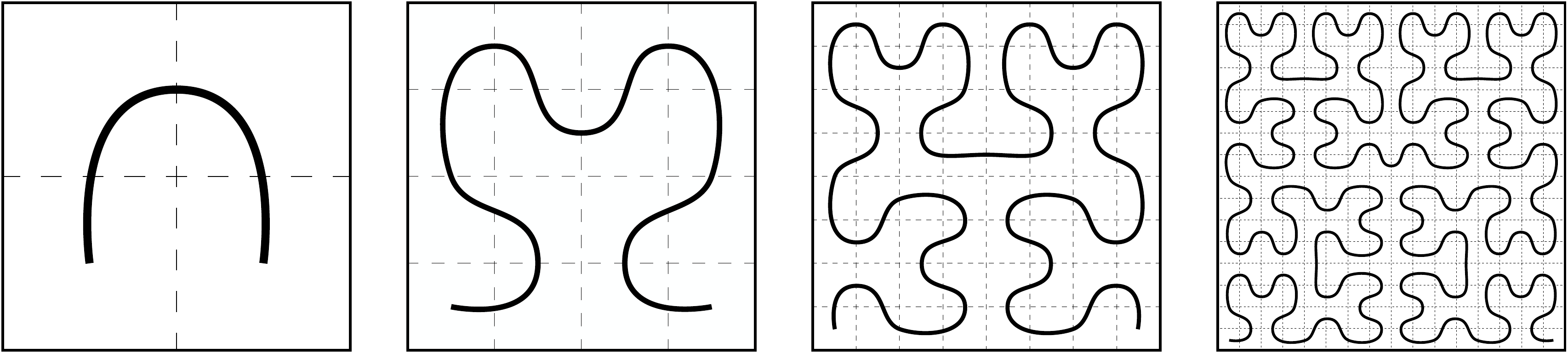}
    \caption{Homogeneous ordering of dyadic partitions on a square}
    \label{fig:square-ordering}
\end{figure}

Let us consider the unit square $\Lambda=[-\frac{1}{2},\frac{1}{2}]^2\subset \RR^2$. If we proceed like we did with the Sierpi\'nski gasket, we find $\Lambda$ as the attractor of the IFS $\{\vphi_1,\vphi_2,\vphi_3,\vphi_4\}$, where, for all $z\in\CC$, 
\begin{gather*}
    \vphi_1(z) = -\frac{1}{2}\Bar{z}e^{-i\frac{\pi}{2}}-\frac{1}{4}-\frac{1}{4}i,\quad \vphi_2(z)=\frac{1}{2}z-\frac{1}{4}+\frac{1}{4}i,\\
    \vphi_3(z)= \frac{1}{2}z+\frac{1}{4}+\frac{1}{4}i \quad \text{and} \quad \vphi_4(z)=-\frac{1}{2}\Bar{z}e^{i\frac{\pi}{2}}+\frac{1}{4}-\frac{1}{4}i.
\end{gather*}
We then notice that, just like in the case of the Sierpi\'nski gasket, the lexicographical ordering found in each resolution corresponds to the path followed by the respective pseudo-Hilbert curve. Therefore, it is more intuitive (albeit less formal) to simply consider, in each resolution $m$,  the $m$th dyadic partition on $\Lambda$ (that it, the subdivision of $\Lambda$ by $4^m$ squares of side $1/2^m$) ordered ``visually'' by the pseudo-Hilbert curve of order $m$ (like in Figure \ref{fig:square-ordering}).

We then immediately get that $\lambda$ has HBD$^\circ$ at most 2. Indeed, the fact that in each step we have a covering is trivial (it is the dyadic covering after all). Furthermore, each step in the construction contracts the previous one with a ration of $\frac{1}{2}$, thus giving \ref{hbd:i}. Homogeneity also comes from the construction of the dyadic covering and the positioning of the sub-squares, this gives \ref{hbd:ii}. Finally, ordering comes from the pseudo-Hilbert curves, so \ref{hbd:iii} follows and the verification is complete.

Notice that the calculation of the HBD of $\Lambda$ is much simpler, for the dyadic partition is a homogeneous covering (see \cite[Section 4.2]{BCQCommon}). The ordering requirement for the HBD$^\circ$ makes the calculation more delicate. Also notice that this construction can be done in any dimension $d\geq 2$. In fact, curves with the same properties as the Hilbert curve can be constructed in any dimension (there are even multiple ways of doing it). See \cite[Section 2.8]{hans} for one construction in dimension 3.

\subsection{Fractal curves generated by a seed}

A simple way of defining a fractal is to fix a line segment as the \emph{base}, which we identify with the \emph{resolution 0 shape} $R_0$, and a small shape in the form of an open polygonal path, which we will call the \emph{seed}. The construction, represented in Figure \ref{seed-3-iterations}, goes by the following iterative steps: (1) scale a copy of the seed so that the distance of its ends equal the length of the base; (2) position the scaled seed end-to-end with $R_0$ (after rotating if necessary); (3) define this duplicated and scaled seed as the \emph{resolution 1 shape} $R_1$; (4) repeat the same process with each line segment composing $R_1$ in order to obtain the \emph{resolution 2 shape} $R_2$; (5) repeat. The limit of $R_n$ as $n$ goes to $+\infty$ is a continuous fractal curve.

\begin{figure}[H]
    \centering
    \includegraphics[width=0.95\textwidth]{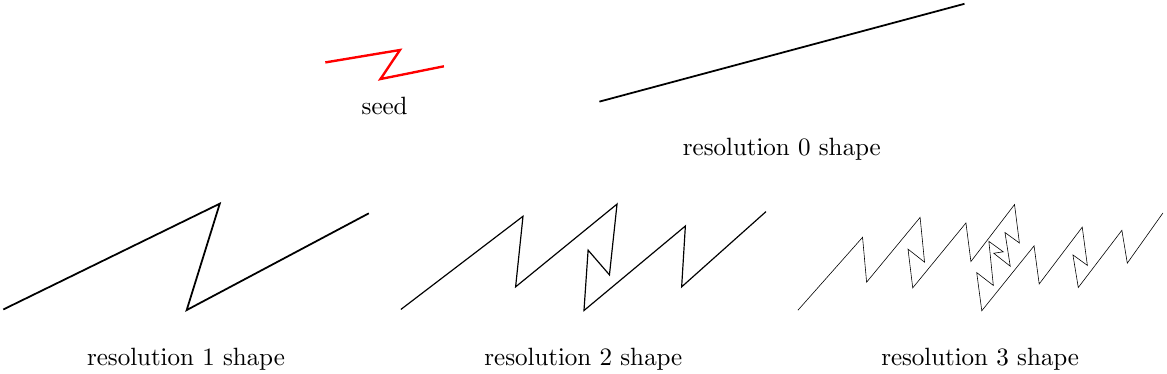}
    \caption{Seed generating a fractal, first 3 iterations}
    \label{seed-3-iterations}
\end{figure}

All the features of the fractal figure obtained through the algorithm above are determined by its seed. In fact, if the seed is an open polygonal path with $r$ edges, then the fractal has $r$ similarities. If each edge has the same length, then the fractal is uniformly contracting, and its contraction ratio eqsuals the length of the edges of the first scaled seed $R_1$ divided by the length of the initial line segment. In this case, the pseudo-curves of each resolution provides a homogeneous family of coverings for this fractal and we can easily conclude that the limit curve has HBD at most its similarity dimension $\gamma=-\frac{\log r}{\log c}$ (where $c$ is its uniform contraction ratio). Furthermore, the ordering can be defined following the pseudo-curve of respective resolution, thus the curve actually have HBD$^\circ$ at most $\gamma$. If, on top of that, the fractal satisfy the OSC, then $\gamma$ is in fact its Hausdorff dimension. We shall discuss optimal applications for such constructions in Section \ref{sec:app}.

Many interesting fractals can be obtained through this iterative process, although depending on the choice of the seed, the limit shape can be very messy. For instance, most seeds will not produce fractals satisfying the OSC. For a more detailed discussion, we refer to \cite[Chapter 3]{LauBook}.

\begin{figure}[H]
    \centering
    \includegraphics[width=0.95\textwidth]{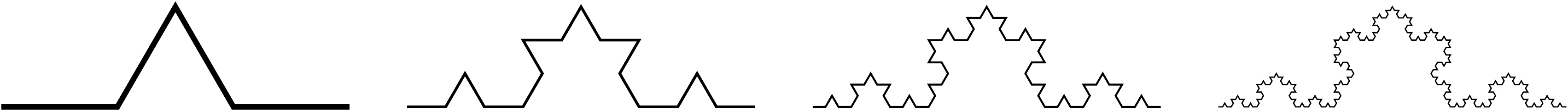}
    \caption{First 4 iterations in the construction of the Koch curve}
    \label{fig:koch:4}
\end{figure}

One example of fractal curve that can be constructed this way is the \emph{Von Koch curve}, which composes the famous \emph{Koch snowflake}. Defined in 1904 by the Swedish mathematician Helge von Koch \cite{koch1904}, it uses as seed the first shape in Figure \ref{fig:koch:4}, thus having 4 similarities. In the classical definition, all edges have length 1/3 of the base, and the angles of the inner part equal 60$^\circ$. Hence, the Koch curve has uniform contraction ratio 1/3. Therefore, it has HBD$^\circ$ at most $\frac{\log(4)}{\log(3)}$, which equals its Hausdorff dimension since the OSC is satisfied.

\begin{figure}[H]
    \centering
    \includegraphics[width=0.95\textwidth]{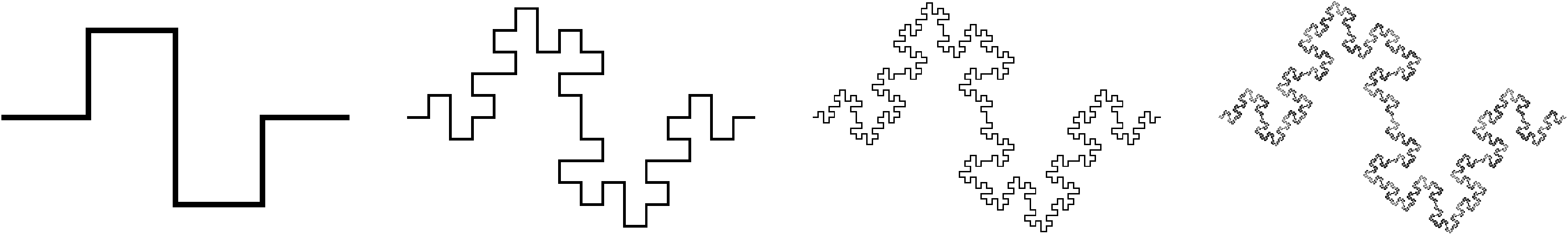}
    \caption{First 4 iterations in the construction of the Minkowski sausage}
    \label{fig:sausage-4}
\end{figure}

Another example of well known fractal curve obtained through the iterative process described above is the so-called \emph{Minkowski sausage} (defined in \cite[Page 33]{mandelbrot}). The first 4 iterations of its construction is represented in Figure \ref{fig:sausage-4}. It uses as seed an 8 edges polygonal path (although we see 7, the longer one in the middle is considered double), hence this fractal has 8 similarities. Each edge of the scaled seed has side 1/4 of the length of its base, thus it is uniformly contracting with ratio 1/4. Therefore, the Minkowski sausage has HBD$^\circ$ at most $\frac{\log(8)}{\log(4)}=\frac{3}{2}$. Once more, it satisfies the OSC, thus having Hausdorff dimension $\frac{3}{2}$ as well.

\subsection{Hölder curves} \label{subsec:holder}

As it is verified in \cite[Sextion 4.2]{BCQCommon}, the image of any $\beta$-Hölder curve, $\beta\in (0,1]$, has HBD at most $1/\beta$. Since we can order the family of coverings following the sense of the curve, we get that it also has HBD$^\circ$ at most $1/\beta$. We include here the details for the sake of completeness.

Let $d\geq 1$, $\beta\in(0,1]$, $\rho>0$, and $f:[0,1]\to \RR^d$ satisfying \[\|f(x)-f(y)\|\leq \rho|x-y|^\beta.\]
We define $c=2^{-\beta}$ and we write $\Gamma=f([0,1])$.

For each $m\geq1$, consider the dyadic intervals $I_{\bf i}$ defined, for ${\bf i}\in \{1,2\}^m$, as \[I_{\bf i}=\bigg[\sum_{k=1}^m\frac{i_{k}-1}{2^k}, \sum_{k=1}^m\frac{i_{k}-1}{2^k}+\frac{1}{2^m}\bigg].\] We define $\Lambda_{\bf i}=f(I_{\bf i})$ for all ${\bf i}\in \{1,2\}^m$. It follows that
\[\diam (\Lambda_i)\leq \rho\diam(I_{\bf i})^\beta=c^{m}\rho.\]
Hence, we can apply here the definition of HBD$^\circ$ to $\Gamma$ by taking $r=2$, $\gamma=1/\beta$ and using the family $\big((\Lambda_{\bf i})_{{\bf i}\in \{1,2\}^m}\big)_{m\geq 1}$ defined above. We conclude that $\Gamma$ has HBD$^\circ$ at most $\gamma$.
\section{Applications} \label{sec:app}

In this section, we provide some consequences and applications of Theorem \ref{cs-gamma-dim}. As we shall discuss, whenever a parameter set $\Lambda$ has HBD$^\circ$ at most $\dim_\H(\Lambda)$, we get optimal applications, that is, we include the previously uncovered limit case (as discussed in \cite{BCQCommon}).

Weighted forward and backward shifts acting on sequence spaces are among the most studied operators in Linear Dynamics. We say that $X$ is a \emph{sequence space} when it is a subspace of $\omega=\CC^{\NN_0}$, thus a \emph{Fréchet sequence space} is a sequence space with the structure of Fréchet space. Given a sequence $w=(w_n)_{n\in\NN}\in \RR_+^\NN$ of positive scalars, the \emph{weighted backward shift induced by $w$} is the operator defined by \[B_w(x_0,x_1,\dots)=(w_1x_1, w_2,x_2,\dots),\]
whereas the \emph{weighted forward shift induced by $w$} is defined by \[F_w(x_0,x_1,\dots)=(0, w_1x_0,w_2,x_1,\dots).\]

Throughout this section, $X$ is a Fréchet sequence space endowed with a family of seminorms $\big(\|\cdot\|_p\big)_{p\in\NN}$ such that the $F$-norm $\|\cdot\|$ defined in $X$ is given by
\[\|x\|=\sum_{p=1}^{+\infty}\frac{1}{2^p}\min(1, \|x\|_p), \ \forall x\in X.\]
We also assume that the space $c_{00}=\Span(e_i)_{i\in\NN_0}$ of sequences with finite support is dense in $X$. Notice that $\|\cdot\|$ is not necessarily homogeneous (see \cite{Rolewicz:book} for more details on $F$-spaces), but we are still allowed to use the inequality \[\|x\cdot u\|\leq (|x|+1)\|u\|, \ \forall u\in X, \ \forall x\in\CC.\]

We consider $I\subset \RR_+$ a compact interval and $\big(w(a)\big)_{a\in I}$ a family of positive weights such that the induced family of weighted backward shifts $(B_{w(a)})_{a\in I}$ acting on $X$ is continuous. Let us fix $d\in\NN$ and equip $X^d$ with the $F$-norm 
\[
\|u\|_\infty=\max\{\|u(1)\|,\dots,\|u(d)\|\}, \ \forall u=(u(1),\dots,u(d))\in X^d.
\]
For each $\lambda=(\lambda(1),\dots,\lambda(d))\in I^d$, we define \[T_{\lambda}:= \bigoplus_{j=1}^d B_{w(\lambda(j))}\quad\text{and}\quad S_\lambda := \bigoplus_{j=1}^d F_{w(\lambda(j))^{-1}},\]
where $F_{w(\lambda_k)^{-1}}$ denotes the weighted forward shift induced by the weights $\big(\frac{1}{w_n(\lambda_k)}\big)_{n\in\NN}$. Thus, we immediately have that $\D:=c_{00}^d$ is dense in $X^d$, the map $T_\lambda$ is an operator on $X^d$ and $T_\lambda\circ S_{\lambda}=Id$, for all $\lambda\in I^d$.
\begin{theorem}\label{app:main}
     Let $\Lambda\subset I^d$ be a compact set with HBD$^\circ$ at most $\gamma>0$ and let $\alpha\in(0,1/\gamma]$. Suppose that there are $C_0, C_1>0$ such that the maps $f_n:I\to \RR$ given by \[f_n(x)=\sum_{k=1}^n\log(w_k(x)), \quad x\in[a,b],\]
for all $n\in\NN$, are $C_0n^\alpha$-Lipschitz and 
     \begin{equation}
         \label{cond:app:F}
         \sum_{k=0}^{+\infty}\bigg\|\frac{\exp(C_1k^\alpha)}{w_{l+1}(x)\cdots w_{l+k}(x)}e_{l+k}\bigg\|<+\infty,\quad \forall l\in\NN_0, \forall x\in I.
\end{equation}
    Then $(T_\lambda)_{\lambda\in\Lambda}$ has a dense $G_\delta$ set of common hypercyclic vectors.
\end{theorem}
\begin{proof}
    Let us start by writing $I=[a,b]$, for some $a\leq b$ in $\RR_+$, and choosing some $D\in \big(0, \frac{C_1}{2C_0}\big]$. Given $u\in \D$, there is $L\in\NN_0$ such that we can write $u=(u(1),\dots,u(d))$ in the form $u(j)=\sum_{l=0}^Lu_l(j)e_l$, for all $j=1,\dots,d$. We fix some large $\kappa\in\NN$ (precise conditions will be given during the proof) and we define 
    \[
    c_k:=(L+1)(\|u\|_\infty + 1)\sup_{x\in[a,b]}\max_{l=0,\dots, L} \bigg\|\frac{\exp(C_1k^\alpha)}{w_{l+1}(x)\cdots w_{l+k}(x)}e_{l+k}\bigg\|.
    \]
    Given $\lambda,\mu\in\Lambda$ such that $\|\lambda-\mu\|\leq D\frac{k^{1/\gamma}}{(n+k)^{1/\gamma}}$ with $n\geq 0$ and $k\geq \kappa$, we have $T_\lambda^{n+k}\circ S_\mu^{n}u=0$, for $\kappa>L$, and
    \begin{align*}
        \|T_\lambda^{n} S_{\mu}^{n+k}u\|_\infty \leq \max_{j=1,\dots,d} \|B_{w(\lambda(j))}^n F_{w(\mu(j))^{-1}}^{n+k}u(j)\|
        = \max_{j=1,\dots,d}\bigg\|\sum_{l=0}^Lu_l(j)F_1F_2e_{l+k}\bigg\|,
    \end{align*}
where 
\[
F_1 = \frac{w_{l+1}(\lambda(j))\cdots w_{l+n+k}(\lambda(j))}{w_{l+1}(\mu(j))\cdots w_{l+n+k}(\mu(j))} 
\quad\text{and}\quad 
F_2=\frac{1}{w_{l+1}(\lambda(j))\cdots w_{l+k}(\lambda(j))}.
\]
Since $f_n$ is $C_0n^\alpha$-Lipschitz, we obtain 
\begin{align*}
    F_1 &= \frac{w_{1}(\lambda(i))\cdots w_{l+n+k}(\lambda(i))}{w_{1}(\mu(i))\cdots w_{l+n+k}(\mu(i))}\times \frac{w_1(\mu(i))\cdots w_l(\mu(i))}{w_1(\lambda(i))\cdots w_l(\lambda(i))}\\
    &\leq\exp(C_0((l+n+k)^{\alpha}+l^\alpha)\|\lambda-\mu\|)\\
    &\leq \exp(2C_0(n+k)^{\alpha}\|\lambda-\mu\|)\\
    &\leq \exp(2C_0Dk^\alpha),
\end{align*}
where the penultimate inequality holds if $\kappa$ (hence $k$) is big enough and the last inequality follows from the hypothesis $\|\lambda-\mu\|\leq D\frac{k^{1/\gamma}}{(n+k)^{1/\gamma}}\leq D\frac{k^{\alpha}}{(n+k)^{\alpha}}$. Altogether, we get
\begin{align*}
\|T_\lambda^{n}S_{\mu}^{n+k}u\|_\infty
    &\leq \max_{j=1,\dots,d}\sum_{l=0}^L\big\|u_l(j)F_1F_2e_{l+k}\big\| \\
    &\leq \sup_{x\in[a,b]}\max_{l=0,\dots, L}(L+1)(\|u\|_\infty+1)\bigg\|\frac{\exp(2C_0Dk^\alpha)}{w_{l+1}(x)\cdots w_{l+k}(x)}e_{l+k}\bigg\|\\
    &\leq c_k.
\end{align*}
This verifies condition \ref{CS1}. Condition \ref{CS2} follows the fact that $f_n$ is $C_0n^\alpha$-Lipschitz. Details are left to the reader.
\end{proof}

\noindent {\bf Observation.} If $X$ is a Banach space, we can substitute \eqref{cond:app:F} by 
\begin{equation}\label{cond:simp:app:F}
\sum_{k=0}^{+\infty}\bigg\|\frac{\exp(C_1k^\alpha)}{w_{1}(x)\cdots w_{k}(x)}e_{k}\bigg\|<+\infty, \quad \forall x\in I.
\end{equation}
Indeed, given $l\in\NN_0$ and $x\in I$, the conclusion follows easily form the continuity of $F_{w(x)^{-1}}^l$ and the fact that \[\frac{\exp(C_1k^\alpha)}{w_{l+1}(x)\cdots w_{l+k}(x)}e_{l+k}=w_1(x)\cdots w_l(x)F_{w(x)^{-1}}^l\bigg(\frac{\exp(C_1k^\alpha)}{w_{1}(x)\cdots w_{k}(x)}e_{k}\bigg).\]

As a consequence (and in the same vein as \cite[Corollary 4.2]{BCQCommon}), we can state the following practical corollary.

\begin{corollary} \label{corol:app}
    Let $X=\ell_p$, $p\in[1,+\infty)$, or $X=c_0$ and let $\Lambda\subset I^d$ be compact set with HBD$^\circ$ at most $\gamma>0$. Assume that there are $\alpha\in (0,1/\gamma]$, $N\in\NN$ and $C_0, C_1, C_2>0$ such that, for all $n\geq N$, the function $x\in I\mapsto \sum_{j=1}^n \log\big(w_j(x)\big)$ is $C_0n^\alpha$-Lipschitz and $\inf_{x\in I}w_1(x)\cdots w_n(x)\geq C_1\exp(C_2n^\alpha)$. Then $(T_\lambda)_{\lambda\in\Lambda}$ has a dense $G_\delta$ set of common hypercyclic vectors.
\end{corollary}
Let us now consider the case where $w_1(x)\cdots w_n(x)\approx \exp(xn^\alpha)$, that is, the products $w_1(x)\cdots w_n(x)$ \emph{behave} like $\exp(xn^\alpha)$, namely there are $C_0'$, $C_0''$, $C_1'$, $C_1''>0$ such that, for all $n\in\NN$ and all $x\in I,$ \[C_1'\exp(C_0'xn^\alpha)\leq w_1(x)\cdots w_n(x)\leq C_1''\exp(C_0''xn^\alpha).\]
Then all the hypothesis of Corollary \ref{corol:app} are satisfied, whence $(T_\lambda)_{\lambda\in\Lambda}$ has a common hypercyclic vector whenever $\Lambda$ has HBD$^\circ$ at most $\gamma>0$ and $\alpha\in(0,1/\gamma]$. Furthermore, we know from \cite[Corollary 3.2]{BCQCommon} that $(T_\lambda)_{\lambda\in\Lambda}$ has no common hypercyclic vector when $\dim_\H(\Lambda)>1/\alpha$. In this case, we get a strong equivalence whenever $\gamma=\dim_{\H}(\Lambda)$.
\begin{corollary}
    Let $X=\ell_p, p\in[1,+\infty)$, or $X=c_0$ and let $\Lambda\subset I^d$ be compact set with HBD$^\circ$ at most $\dim_\H(\Lambda)$. If $w_1(x)\cdots w_n(x)\approx \exp(xn^\alpha)$, then $(T_\lambda)_{\lambda\in\Lambda}$ is common hypercyclic if, and only if, $\alpha\in \big(0,1/\dim_\H(\Lambda)\big]$.
\end{corollary}
This immediately implies Corollary \ref{corol:intro}, for any $\beta$-Hölder curve has HBD$^\circ$ at most $1/\beta$ (see Section \ref{subsec:holder}). By considering the HBD$^\circ$ that we have calculated in Section \ref{sec:some-fractals}, we get the following other examples.
\begin{example}
    Consider a compact parameter set $\Lambda\subset \RR_+^2$ and the weights $w_n(x)=1+\frac{x}{n^{1-\alpha}}, x>0, n\in \NN.$ Define the family $(T_\lambda)_{\lambda}$ by $T_\lambda= B_{w(x)}\times B_{w(y)}$ for each $\lambda=(x,y)\in\Lambda$. We have the following.
    \begin{itemize}
        \item If $\Lambda$ is a Sierpiński gasket, then $\bigcap_{\lambda\in \Lambda}HC(T_\lambda)\neq \varnothing \Longleftrightarrow 0<\alpha\leq \frac{\log 2}{\log 3}.$
        \item If $\Lambda$ is a square, then $\bigcap_{\lambda\in \Lambda}HC(T_\lambda)\neq \varnothing \Longleftrightarrow 0<\alpha\leq \frac{1}{2}.$
        \item If $\Lambda$ is a Von Koch curve, then $\bigcap_{\lambda\in \Lambda}HC(T_\lambda)\neq \varnothing \Longleftrightarrow 0<\alpha\leq \frac{\log 3}{\log 4}.$
        \item If $\Lambda$ is a Minkowski sausage, then $\bigcap_{\lambda\in \Lambda}HC(T_\lambda)\neq \varnothing \Longleftrightarrow 0<\alpha\leq \frac{2}{3}.$
    \end{itemize}
\end{example}
For many applications, in particular for the square, the analogous conclusion can be obtained in $d$ dimensions for any $d\in\NN$. Thus, Theorem \ref{cs-d-dim} is a direct consequence of Theorem \ref{cs-gamma-dim}.

It is worth noticing that Corollary \ref{corol:intro} gives an optimal result whenever the fractal $\Lambda$ indexing the family $(T_\lambda)_{\lambda\in\Lambda}$ admits an \emph{optimal parametrization}, namely $\Lambda=f([0,1])$, where $f:[0,1]\to \RR_+^d$ is an $\frac{1}{\dim_\H(\Lambda)}$-Hölder continuous curve. It is folklore that some classical fractals, like the Sierpiński gasket (seen as the Sierpiński arrowhead curve) and the square (seen as the Hilbert curve), satisfy this condition, but there are some surprising examples where this is also possible. For instance, as shown in \cite[Example 5.2]{gifs2}, the Sierpiński carpet can be uniformly ordered and admits an optimal parametrization (see also \cite{gifs1} for more details on the theory of \emph{graph-directed iterated function systems}).

We can also apply Theorem \ref{app:main} to the family $\big(\bigoplus_{k=1}^d \lambda(k) D\big)_{(\lambda(1),\dots, \lambda(d))\in \Lambda}$ when $\Lambda$ is a Lipschitz curve contained in $(e,+\infty)^d$, where $D:f\mapsto f'$ is the operator of complex differentiation acting on the space $H(\CC)$ of entire functions. Indeed, we can interpret $H(\CC)$ as a Fréchet sequence space by identifying each $f\in H(\CC)$ with the sequence $(a_n)_n$ of its Taylor coefficients at 0. The family of seminorms $(\|\cdot\|)_{q\geq 1}$ given by $\|(a_n)_n\|_q=\sum_{n=0}^{+\infty}|a_n|q^n$ induces the topology of uniform convergence on compact sets of $H(\CC)$. By doing so, the operators $xD$, $x>0$, correspond to the weighted backward shifts induced by $w_n(x)=xn$, $n\in\NN$. However, it is possible to get a more general conclusion by applying Theorem \ref{cs-gamma-dim} directly.
\begin{example}
For any $d\geq 1$ and any Lipschitz curve $\Lambda\subset \RR^d_+$, the family $\big(\bigoplus_{k=1}^d \lambda(k) D\big)_{(\lambda(1),\dots, \lambda(d))\in \Lambda}$ has a common hypercyclic vector.
\end{example}
\begin{proof}
Proceed like in \cite[Theorem 13]{CostakisSambarino}.
\end{proof}
\section{Covering homogeneously ordered self-similar fractals} \label{sec:covering}

\subsection{How these coverings actually look like}

Before proceeding to the proof of Lemma \ref{key:lemma}, let us explain with examples how the construction works. Take for instance the Sierpiński gasket. It has $r=3$ similarities and $c=\frac{1}{2}$ uniform contraction ratio, so its similarity (and Hausdorff) dimension is $\gamma=\frac{\log 3}{\log 2}$. Therefore, in the CS-Criterion \ref{cs-gamma-dim}, we require a covering with fineness $\frac{\tau}{n^{1/\gamma}}$ so that property \ref{CS2} can be used. Also notice that, in the proof of the criterion, we have taken $n_i=iN$ for some big $N$. In order words, our covering must be composed by squares of side at most $\frac{\tau}{(N)^{1/\gamma}}, \frac{\tau}{(2N)^{1/\gamma}}, \frac{\tau}{(3N)^{1/\gamma}}, \dots$, and so on. Let us denote these squares by $S_1, S_2, S_3, \dots,$ respectively, and call their side length the \emph{sizes of the ``allowed'' square parts}.
\begin{figure}[H]
    \centering
    \includegraphics[width=0.95\textwidth]{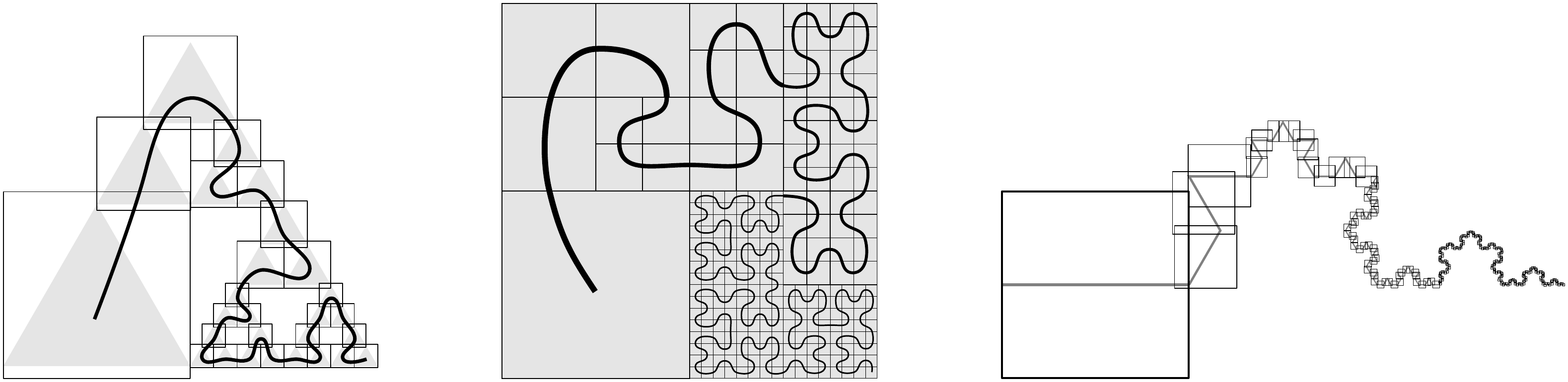}
    \caption{Some examples of refining partitions (for $s=1$)}
    \label{ex:partitions}
\end{figure}
In general, we choose $s$ big enough such that $\frac{\tau}{(N)^{1/\gamma}}\geq c^s=\frac{1}{2^s}$ and we reset $\tau,N$ so that $\frac{\tau}{(N)^{1/\gamma}}=c^s$. For the sake of clarity in the discussion, let us suppose that $s=1$, that is, let us assume $\frac{\tau}{(N)^{1/\gamma}}=c=\frac{1}{2}$ (see Figure \ref{ex:partitions}). In this case, $S_1$ covers the first covering part of the resolution 1 covering (that is, the covering associated to $m=1$ in the definition of HBD$^\circ$). This means that, if we look at the covering of resolution 2, $S_1$ have covered $r=3$ our of $r^2=3^2$ covering parts, so that there are $r^2-r=r(r-1)$ resolution 2 parts left to be covered. This value is divisible by $r-1$ (in general we use that $r(r^{s}-1)$ is divisible by $r-1$). Thus, we group the remaining resolution 2 covering parts 2 by 2. Now, it is a matter of noticing that the sizes of the allowed square parts match the increasing resolution in each group. The fundamental property used is the fact that, whenever we consider allowed square parts up to a power of $r$, we recuperate a side length which matches the contracted side of respective resolution. In fact, $\frac{\tau}{(r^mN)^{1/\gamma}}=\frac{1}{(r^m)^{1/\gamma}}\frac{\tau}{(N)^{1/\gamma}}=\frac{1}{2^{m+1}}=c^{m+1}$, since $r^{1/\gamma}=3^{\log 2/\log 3}=2=\frac{1}{c}$. We have the following conclusions.
\begin{itemize}
    \item $S_2, S_3$ have sizes $\frac{\tau}{(2N)^{1/\gamma}}>\frac{\tau}{(3N)^{1/\gamma}}=c^2$, thus 2 parts of resolution 2 can be covered. There are $4$ resolution 2 parts left to be covered. Notice that $S_3$ is a perfect fit, that is, its side has length $c^2$.
    \item $S_4, S_5, \dots, S_9$ have sizes $\frac{\tau}{(4N)^{1/\gamma}}>\cdots>\frac{\tau}{(9N)^{1/\gamma}}=c^3$. With $S_4, S_5, S_6$ we can cover one part of resolution $2$ and with $S_7,S_8, S_9$ we can cover another part. In total, 2 parts of resolution 2 can be covered. There are $2$ resolution 2 parts left to be covered. Notice that $S_9$ is a perfect fit.
    \item $S_{10}, \dots, S_{27}$ have sizes $\frac{\tau}{(10N)^{1/\gamma}}>\cdots>\frac{\tau}{(27N)^{1/\gamma}}=c^4$. Notice that, altogether, they can cover 2 parts of resolution 2, and this completes the covering. Once more, $S_{27}$ is a perfect fit.
\end{itemize}
We insist in mentioning that we have a ``perfect fit'' at the end of each group because this is what grant optimality of the covering. In fact, although we cover multiple times several small portions of the fractal, the refinement of the final covering matches its contraction ratio at the end of each patch. Two other examples are also represented in Figure \ref{ex:partitions}.

The following section is dedicated to the general construction of the covering. Thus, it can be seen as a long proof of our Lemma \ref{key:lemma}. We will establish several notations and definitions throughout the construction. We will also state and prove a couple of useful smaller lemmas.

\subsection{Proof of Lemma \ref{key:lemma}} \label{sec:lemma:proof}

For all that follows, $\Gamma$ is a compact subset of $\RR^d$ for some $d\geq 1$, with homogeneously ordered box dimension at most $\gamma\in (0,d]$. For simplicity we will make the construction for $d=2$. Let $r\geq 2$ and $\rho>0$ be such that, for all $m\geq 1$, there is a covering of $\Gamma$ by compact sets $(\Lambda_{\bf i})_{{\bf i}\in I_r^m}$ satisfying conditions (i), (ii) and (iii) of the definition of HBD$^\circ$. We can assume that all these covering parts are squares. We set $\alpha=1/\gamma$, $c=1/r^\alpha$ and we fix $D\geq \frac{\rho}{c^3}$ arbitrarily.

The \emph{portion of $\Gamma$ inside a rank $s$ part $\Lambda_{\bf i}$} is defined to be \[\Gamma({\bf i})=\bigg(\bigcap_{m\geq1}\bigcup_{{\bf j} \in I_r^m}\Lambda_{{\bf i},{\bf j}}\bigg)\cap \Gamma,\]
where ${\bf i}\in I_r^s$. We have, for all $s\in \NN$, \[\Gamma=\bigcup_{{\bf i}\in I_r^s}\Gamma({\bf i}).\]
Each of these portions are themselves sets with HBD$^\circ$ at most $\gamma$ for the family of coverings $\big((\Lambda_{{\bf i},{\bf j}})_{{\bf j}\in I_r^m}\big)_{m\geq 1}$. Similarly, we can define the \emph{portion of $\Gamma$ outside a rank $s$ part $\Lambda_{\bf i}$} as $\bigcup_{{\bf j}\neq {\bf i}}\Gamma({\bf j})$ and the \emph{portion of $\Gamma$ inside a subfamily of parts}, say $\{\Lambda_{\bf j}\}_{{\bf j}\succ{\bf i}}$, as the union $\bigcup_{{\bf j}\succ {\bf i}}\Gamma({\bf j})$ (remember that $\prec $ stand for the lexicographical order).

We can assume without loss of generality that $\rho/c^3\leq D$ (otherwise we repeat the construction a finite number of times on portions of $\Gamma$ inside the parts of a covering of resolution some sufficiently large $m$).

For each resolution $m\in\NN$ and each ${\bf k}\in I_r^m$, we let $\lambda_{\bf k}=(x_{\bf k}, y_{\bf k})$ be the bottom-left corner of the square $\Lambda_{\bf k}$. From property (i) of the covering $(\lambda_{\bf i})_{{\bf i}\in I_r^m}$, we know that \[\Lambda_{\bf k}\subset \big[x_{\bf k}, x_{\bf k}+c^m\rho\big]\times\big[y_{\bf k}, y_{\bf k}+c^m\rho\big].\]
The first step in our construction is to establish Lemma \ref{dSlemma:general}, which allows us to count the indexes ``between'' two parts of same ``resolution''. Let us first establish some notations and definitions to simplify our statements and discussion. Let us define what we mean by \emph{resolution} or \emph{rank}, \emph{jump size} and \emph{enumeration size}. For each $m\in\NN$, we known that $(\lambda_{\bf i},\Lambda_{\bf i})_{{\bf i}\in I_r^m}$ is a tagged covering of $\Gamma$. For each ${\bf i}\in I_r^m$, we call $m$ the \emph{rank} (or \emph{resolution}) of the \emph{covering part} $\Lambda_{\bf i}$. Given ${\bf i}=(i_1,\dots,i_m), {\bf j}=(j_1,\dots,j_m)\in I_r^m$, we talk about where a \emph{jump} $\lambda_{\bf i}\to\lambda_{\bf j}$ happens to emphasize the subdivision parts $\Lambda_{\bf i}$ and $\Lambda_{\bf j}$ containing $\lambda$ and $\mu$, respectively. When we say, for example, that the \emph{jump} $\lambda\to\mu$ happens in the same rank $m-1$ covering part, we mean that $i_k=j_k$ for all $k=1,\dots,m-1$, in other words, $\lambda,\mu\in \Lambda_{i_1,\dots,i_{m-1}}$. We also talk about the \emph{size of the enumeration from ${\bf i}$ to ${\bf j}$} to emphasize the number of indexes that are counted from ${\bf i}$ to ${\bf j}$ in the lexicographical order. This number is denoted by ${\bf i}\to{\bf j}$. We also talk about the \emph{size of the jump} $\lambda_{\bf i}\to\lambda_{\bf j}$ to refer to the number $\|\lambda_{\bf l}-\lambda_{\bf j}\|$. Lemma \ref{dSlemma:general} is fundamental, as it associates the value ${\bf i}\to{\bf j}$ to $\lambda_{\bf i}\to\lambda_{\bf j}$ whenever they are indexes corresponding to covering parts of same resolution. We finish our setup by pointing out a simple (but noticeable) feature of the way we tagged our family of coverings. If $\lambda_{\bf i},\lambda_{\bf j}$ are two tags, not necessarily in the same resolution, then, by setting $k=\diam\big(\Lambda_{\bf i}\cup \Lambda_{\bf j}\big)$, the following \textbf{strict} inequality holds. \[\|\lambda_{\bf i}-\lambda_{\bf j}\|<k.\]
This is a consequence of the fact that the tags $\lambda_{\bf i},\lambda_{\bf j}$ are defined to be the bottom-left corner of the squares $\Lambda_{\bf i}, \Lambda_{\bf j}\subset \Lambda_{\bf i}\cup \Lambda_{\bf j}$, so, precisely, \[\|\lambda_{\bf i}-\lambda_{\bf j}\|\leq \min\big(k-\diam(\Lambda_{\bf i}),k-\diam(\Lambda_{\bf j})\big).\]

\begin{lemma}\label{dSlemma:general}
    For all $m\in\NN$ and all ${\bf j}, {\bf l}\in I_r^m$, 
    \begin{equation*}
        n\in[\![0,m[\![, \ \|\lambda_{\bf l}-\lambda_{\bf j}\|\geq \frac{c^m\rho}{c^n}\implies {\bf j}\to{\bf l}\geq \frac{r^{n-1}+r-2}{r-1}.
    \end{equation*}
\end{lemma}
\begin{proof}
Let $m\in\NN$, ${\bf j}=(j_1,\dots,j_m), {\bf l}=(l_1,\dots,l_m)\in I_r^m$ and $n\in[\![0,m[\![$ as in the hypothesis. Writing $k=m-n\geq 1$, we have \[\|\lambda_{\bf l}-\lambda_{\bf j}\|\geq \frac{c^m\rho}{c^n}=c^k\rho=\rho\bigg(\frac{1}{r^{1/\gamma}}\bigg)^k,\]
which is at least the diameter of the resolution $k$ parts of the family of coverings of $\Gamma$. 
Hence, from ${\bf j}$ to ${\bf l}$, the enumeration has to leave one resolution $k$ part and go to the following one. We notice that, since $0<c<1$, the value $c^k\rho$ is strictly bigger than any resolution $k+1$ part (which have diameter $\leq c^{k+1}\rho$). We also notice that inside each resolution $k$ part there are at least 2 resolution $k+1$ parts (because $r\geq 2$). We conclude that the enumeration ${\bf j}\to{\bf l}$ has to count at least one full rank $k+1$ subdivision part, which contains $r^{m-(k+1)}$ indexes inside, that is, 
\[{\bf j}\to{\bf l}\geq r^{m-(k+1)}=r^{n-1}.\]
Some examples, with and without overlaps, are represented in Figure \ref{kiwi}.
\begin{figure}[H]
\begin{subfigure}{.32\textwidth}
  \centering
  \includegraphics[width=4cm]{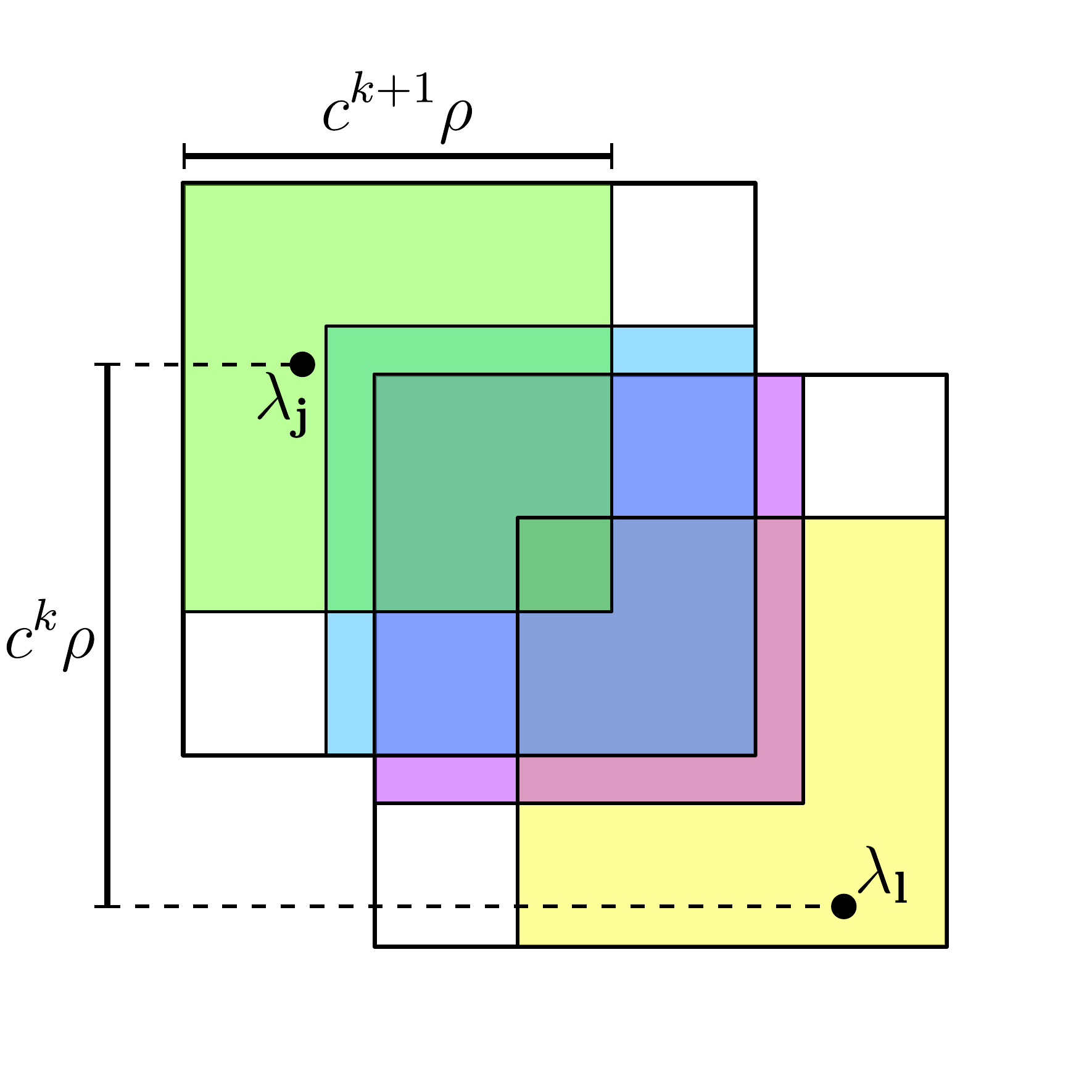}
  \caption{Big overlap}
  \label{kiwi-big}
\end{subfigure}%
\begin{subfigure}{.32\textwidth}
  \centering
  \includegraphics[width=4cm]{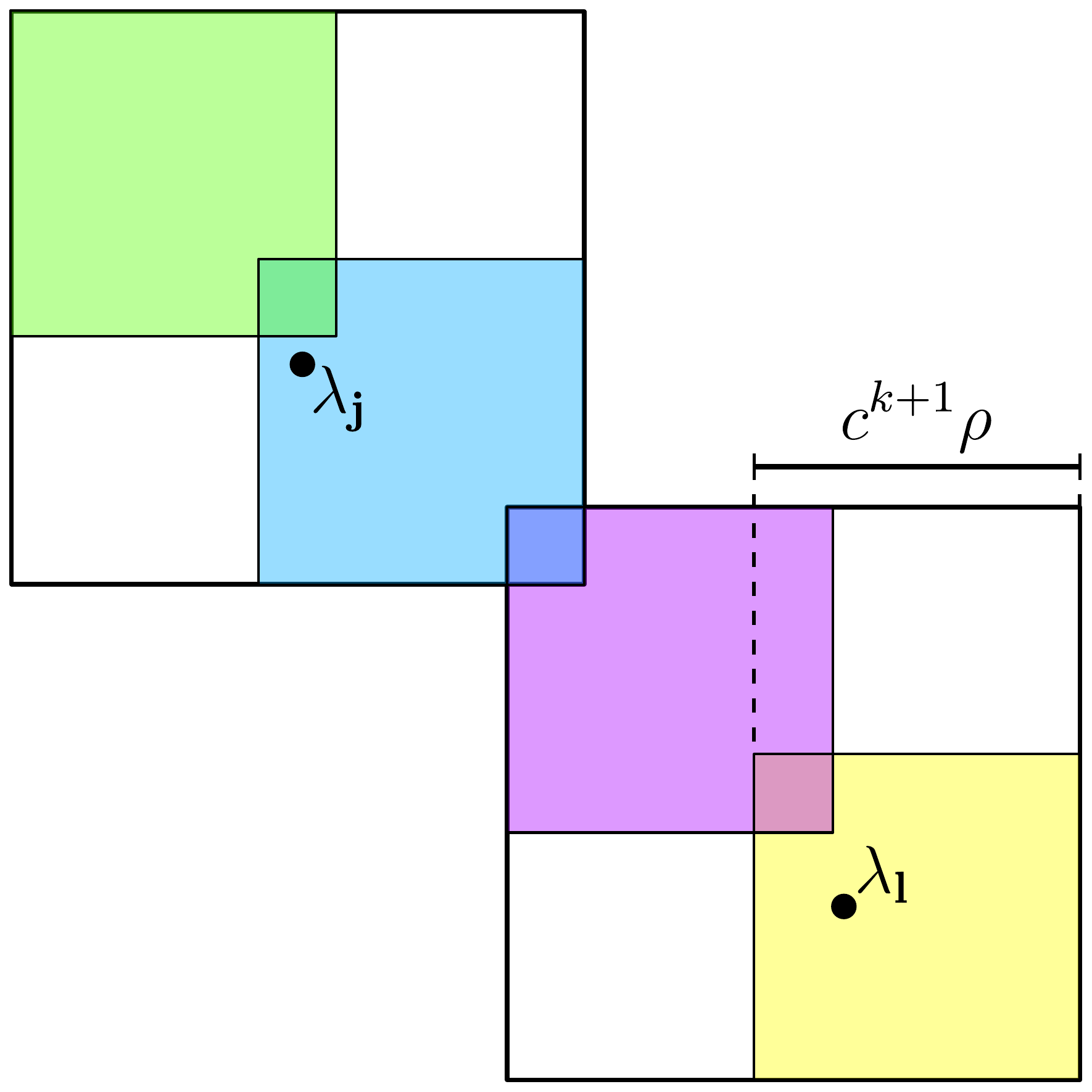}
  \caption{Small overlap}
  \label{kiwi-medium}
\end{subfigure}%
\begin{subfigure}{.32\textwidth}
  \centering
  \includegraphics[width=4cm]{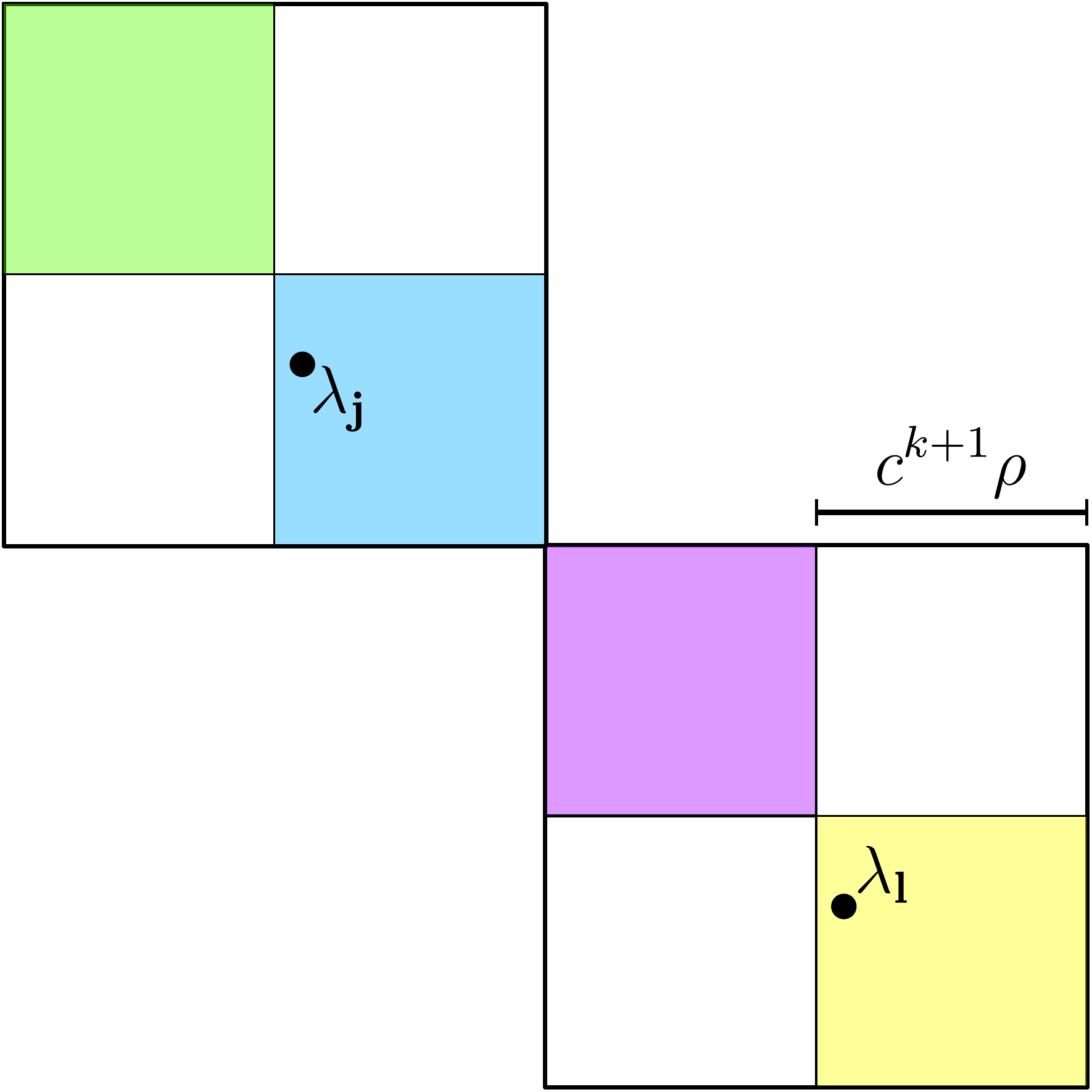}
  \caption{No overlap}
  \label{no-kiwis}
\end{subfigure}
\caption{
Comparing the index counting on all kinds of overlaps. In any case, one rank $k+1$ subdivision part (represented in purple in these examples) is fully counted.
}
\label{kiwi}
\end{figure}
If $n=0$ then obviously \[{\bf j}\to{\bf l}\geq 1\geq \frac{r^{n-1}+r-2}{r-1}.\]
If $n\geq1$, then
\[{\bf j}\to{\bf l}\geq r^{n-1}\geq \frac{r^{n-1}+r-2}{r-1}.\]
This completes the proof.
\end{proof}

Let us continue the proof of Lemma \ref{key:lemma}. Given $\tau>0$ and $N\in\NN$, we aim to construct a tagged covering $(\lambda_k,\Gamma_k)_{k=1,\dots,q}$ of $\Gamma$, for some $q\in\NN$, of the form 
\[
\Gamma_k =\biggl[x_k, x_k+\frac{\tau}{(kN)^{1/\gamma}}\biggr]\times\biggl[y_k, y_k+\frac{\tau}{(kN)^{1/\gamma}}\biggr]
 \quad 
k=1,\dots,q,
\]
where $\lambda_1=(x_1,y_1),\dots, \lambda_q=(x_q,y_q)$, and satisfying
\[
\|\lambda-\mu\|\leq D\Big(\frac{l-j}{l}\Big)^{1/\gamma} = D\Big(\frac{n_l-n_j}{n_l}\Big)^{1/\gamma},
\] for all $1\leq j<l\leq q$ and all $\lambda\in \Gamma_j, \mu\in\Gamma_l$. This construction will be made in several steps. To begin with, let $s\in\NN$ be such that $\frac{\tau}{N^\alpha}\geq c^s\rho$. We can assume without loss of generality that $\frac{\tau}{N^\alpha}=c^s\rho$ (by decreasing $\tau$ if necessary) and $3(r-1)^\alpha c^s\leq 1$ (by increasing $s$ if necessary). We shall make use of the following lemma, whose verification is trivial.
\begin{lemma}\label{lemma:side:general}
    For all $k,n\in\NN$, if $\frac{\tau}{(kN)^{\alpha}}=c^n\rho,$ then \[\frac{\tau}{((k+1)N)^{\alpha}}>\frac{\tau}{((k+2)N)^{\alpha}}>\cdots > \frac{\tau}{(rkN)^{\alpha}}=c^{n+1}\rho.\]
\end{lemma}

We first look at $\Gamma$ at resolution $s$, so the covering parts have diameter at most $c^s\rho$. Thus, one square of side $c^s\rho$ can cover the first rank $s$ part $\Lambda_{1,\overset{s}{\dots},1}$. Let 
\[\Gamma_1=[x_1,x_1+c^s\rho]\times[y_1,y_1+c^s\rho]\supset \Lambda_{1,\overset{s}{\dots},1},\]
where the tag $\lambda_1= (x_1,y_1)$ is ``inherited'' from the part $\Lambda_{1,\overset{s}{\dots},1}$ that $\Gamma_1$ covers, that is, $x_1=x_{1,\overset{s}{\dots},1}$ and $y_1=y_{1,\overset{s}{\dots},1}$. We also define the (only) rank $s$ subdivision part $H_s^{1}(1)=\Gamma_1$. This part encompasses nothing but one index $j=1$, so it has \emph{fineness} $1$. The forthcoming definitions will further clarify what we mean by ``fineness''.

If we look at what we have covered of $\Gamma$ so far, we see that only the portion of $\Gamma$ inside $\Lambda_{1,\overset{s}{\dots},1}$ was covered, that is, there is the portion of $\Gamma$ inside the $r^s-1$ rank $s$ parts $\{\Lambda_{\bf j}\}_{{\bf j}\succ (1,\overset{s}{\dots},1)}$ left to be covered. Let us increase the resolution by one, that is, we now look at each member from $\{\Lambda_{\bf j}\}_{{\bf j}\succ (1,\overset{s}{\dots},1)}$ as containing $r$ rank $s+1$ subdivision parts each. In order words, we have left to be covered the portion of $\Gamma$ inside the $r(r^s-1)$ rank $s+1$ subdivision parts $\{\Lambda_{\bf j}\}_{{\bf j}\succ (1,\overset{s}{\dots},1,r)}$. To cover these $r(r^s-1)$ parts is our goal until the end of this construction. They have diameter at most $c^{s+1}\rho$, so we can apply Lemma \ref{lemma:side:general} in order to ensure that the first $r-1$ of these parts can be covered by squares of side $\frac{\tau}{(2N)^\alpha}>\cdots > \frac{\tau}{(rN)^\alpha}=c^{s+1}\rho$. We write these squares as $\Gamma_2,\dots,\Gamma_r$, where, for $j=2,\dots,r$,
\[\Gamma_j=\bigg[x_j, x_j+\frac{\tau}{(jN)^\alpha}\bigg]\times\bigg[y_j, y_j+\frac{\tau}{(jN)^\alpha}\bigg]\supset [x_j,x_j+c^{s+1}\rho]\times [y_j,y_j+c^{s+1}\rho].\]
Their tags, noted $\lambda_2=(x_2,y_2),\dots,\lambda_r=(x_r,y_r)$, are naturally inherited from the resolution $s+1$ parts they respectively cover. We now define the rank $s+1$ subdivision parts $H_{s+1}^{1}(k)=\Gamma_{k+1}$ for $k=1,\dots,r-1$. Each part $H_{s+1}^{1}(k)$ encompasses the only index $k+1$, so they have \emph{fineness} $1$. We have $r-1$ of them and they cover one rank $s+1$ part each. In total, they cover $r-1$ rank $s+1$ parts.

Looking at what we have covered so far, we have left to be covered the portion of $\Gamma$ inside $r(r^s-1)-(r-1)$ rank $s+1$ parts. We increase the resolution once more, looking at this portion of $\Gamma$ at resolution $s+2$. The parts now have diameter at most $c^{s+1}\rho$, so we can apply Lemma \ref{lemma:side:general} again in order to ensure that any square of side $\frac{\tau}{((r+1)N)^\alpha}>\cdots > \frac{\tau}{(r^2N)\alpha}=c^{s+2}\rho$ can cover one of these parts. In total, $r^2-(r+1)+1=r^2-r=r(r-1)$ parts of resolution $s+2$ can be covered. Let us cover the first $r(r-1)$ of the parts in this portion of $\Gamma$ by tagged squares $\Gamma_{r+1}, \dots, \Gamma_{r^2}$ where, for $j=r+1,\dots, r^2$,
\[\Gamma_j=\bigg[x_j, x_j+\frac{\tau}{(jN)^\alpha}\bigg]\times\bigg[y_j, y_j+\frac{\tau}{(jN)^\alpha}\bigg]\supset [x_j,x_j+c^{s+2}\rho]\times [y_j,y_j+c^{s+2}\rho].\]
Once again, their tags, noted $\lambda_{r+1}=(x_{r+1},y_{r+1}),\dots,\lambda_r=(x_{r^2},y_{r^2})$, are naturally inherited from the resolution $s+2$ parts they respectively cover. We now define the subdivision parts of rank $s+2$ and fineness $1$ as $H_{s+2}^1(k)=\Gamma_{k+r}$, for $k=1,\dots, r(r-1)$. Each of these subdivision parts encompasses 1 index, what explain their fineness being 1. Since there are $r(r-1)$ of these, we can make $r-1$ groups of $r$ rank $s+2$ subdivision parts each. Let us name these groups in order. The first $r$ rank $s+2$ subdivision parts $H_{s+2}^1(1),\dots,H_{s+2}^1(r)$ will be grouped in one rank $s+1$ subdivision part $H_{s+1}^r(1)$, which now have fineness $r$ because it encompasses $r$ indexes inside, that is, the set \[H_{s+1}^r(1):=\bigcup_{k=1}^r H_{s+2}^1(k)\]
contains the $r$ covering parts $\Gamma_{r+1}, \dots, \Gamma_{2r}$. Notice that these $r$ covering parts cover the portion of $\Gamma$ inside one resolution $s+1$ part. Analogously, we define 
\[H_{s+1}^r(n)=\bigcup_{k=1}^r H_{s+2}^1(k+(n-1)r), \quad n=2, \dots, r-1,\]
and we notice that each one of the subdivision parts $H_{s+1}^r(2),\dots,H_{s+1}^r(r-1)$ of rank $s+1$ and fineness $r$ cover the portion of $\Gamma$ inside one rank $s+1$ part. There are $r-1$ of them, thus, altogether they cover $r-1$ rank $s+1$ parts.

\begin{figure}
    \centering
    \includegraphics[width=0.95\textwidth]{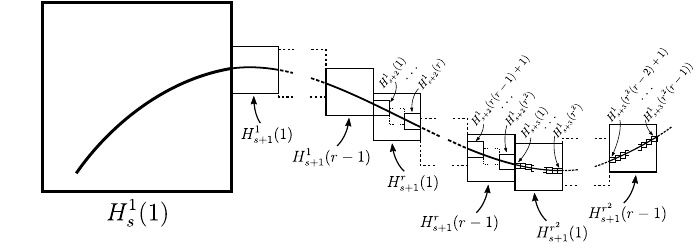}
    \caption{Covering progress up to the last rank $s+1$ subdivision part of fineness $r^2$.}
    \label{h-s-groups}
\end{figure}

Looking at what we have covered until now, we have left to be covered the portion of $\Gamma$ inside $r(r^s-1)-(r-1)-(r-1)=r(r^s-1)-2(r-1)$ rank $s+1$ parts. We increase the resolution once more, that is, we look at this portion of $\Gamma$ at resolution $s+3$. The parts now have diameter at most $c^{s+3}\rho$, so we can apply Lemma \ref{lemma:side:general} in order to ensure that any square of side $\frac{\tau}{((r^2+1)N)^\alpha}>\cdots > \frac{\tau}{(r^3N)^\alpha}=c^{s+3}\rho$ can cover one of these rank $s+3$ parts. In total, $r^3-(r^2+1)-1=r^2(r-1)$ parts of resolution $s+3$ can be covered. Let us cover the first $r^2(r-1)$ rank $s+3$ parts of this portion of $\Gamma$ by tagged squares $\Gamma_{r^2+1}, \dots, \Gamma_{r^3}$, where, for $j=r^2+1,\dots,r^3$, \[\Gamma_{j}=\bigg[x_j, x_j+\frac{\tau}{(jN)^\alpha}\bigg]\times\bigg[y_j, y_j+\frac{\tau}{(jN)^\alpha}\bigg]\supset [x_j,x_j+c^{s+3}\rho]\times [y_j,y_j+c^{s+3}\rho]\]
and their tags where inherited form the resolution $s+3$ parts they respectively cover. As before, the rank $s+3$ subdivision parts with fineness 1 are defined by $H_{s+3}^1(k)=\Gamma_{k+r^2}$, for $k=1,\dots, r^2(r-1)$. With them we can make $r(r-1)$ groups each one containing $r$ parts, these groups are noted $H_{s+2}^r(1), \dots, H_{s+2}^r(r(r-1))$ and defined as \[H_{s+2}^r(n)=\bigcup_{k=1}^{r} H_{s+3}^1(k+(n-1)r), \quad n=1,\dots,r(r-1).\]
Each $H_{s+2}^r(n)$ has rank $s+2$, has fineness $r$ and cover one resolution $s+2$ part. Since there are $r(r-1)$ of these, we can make $r-1$ groups with $r$ parts each. These groups are noted $H_{s+1}^{r^2}(1),\dots,H_{s+1}^{r^2}(r-1)$ and defined by \[H_{s+1}^{r^2}(n)=\bigcup_{k=1}^{r} H_{s+2}^r(k+(n-1)r), \quad n=1,\dots,r-1.\]
Each $H_{s+1}^{r^2}(n)$ has rank $s+1$, fineness $r^2$ and cover one resolution $s+1$ part. Since there are $r-1$ of these, we can cover $r-1$ rank $s+1$ parts in total. See Figure \ref{h-s-groups} for a representation of the covering progress up to this point.

Looking at we have covered so far, we have left to be covered the portion of $\Gamma$ inside $r(r^s-1)-2(r-1)-(r-1)=r(r^s-1)-3(r-1)$ rank $s+1$ parts. If we continue this process until we define subdivision parts with rank up to $s+t$, we will have left to be covered $r(r^s-1)-t(r-1)$ rank $s+1$ parts. Hence, we just need to go until $t=\frac{r(r^s-1)}{r-1}$ (notice that $r^s-1$ is always divisible by $r-1$) in order to have 0 rank $s+1$ parts left to be covered, that is, $(\lambda_k,\Gamma_k)_{k=1,\dots, r^t}$ is a covering of the whole set $\Gamma$.

{\bf Remark.} \emph{Notice that the construction of this covering is possible for any set with homogeneous box dimension at most $\gamma$, not necessarily ordered. Thus, up to this point in our construction, the covering represents an interesting fact by itself, and could be made in a simpler way if one doesn't need ordering. However, for common hypercyclicity, ordering is very important, for it allows us to verify property \eqref{cerise}, as we shall see below.}

Before continuing, let us briefly talk about some important properties of the sets $H_{s+j}^{r^p}(n)$, which are rank $s+j$ subdivision parts of fineness $r^p$, $n=1,\dots,r-1$, $j,p=0,\dots,t$. The fineness of these sets refer to the number of members from $(\lambda_k,\Gamma_k)_{k=1,\dots, r^t}$ that they contain. The first good feature of  $H_{s+j}^{r^p}(n)$ is that the $r^p$ sets from $(\Gamma_k)_{k=1,\dots, r^t}$ that it contains are parts from $\{(\Lambda)_{{\bf i}\in I_r^m}\}_{m\in\NN}$ of same resolution. This comes from the construction. Another feature is that its fineness can be used to estimate, with good precision, the number of indexes counted from some $\Gamma_j$ to some $\Gamma_l$ by looking at the subdivision parts they belong to. To this end, Lemma \ref{dSlemma:general} will be very useful. Fineness also allows us to get a good bound for $l$ whenever we know to which rank $s+1$ subdivision part it belongs to. In fact, it is easy to check that, if $\lambda_l$ belongs to a rank $s+1$ subdivision of fineness $r^k$, then $l\leq r^{k+1}$. Indeed, we have that
\begin{itemize}
    \item $H_s^1(1)$ has 1 index $\longrightarrow$ total $=1$;
    \item $H_{s+1}^1(n)$ has 1 index for each $n=1,\dots, r-1\longrightarrow$ total $=1+(r-1)$;
    \item $H_{s+1}^r(n)$ has r indexes for each $n=1,\dots, r-1 \longrightarrow$ total $=1+(r-1)+r(r-1)$;
    \item $\vdots$
    \item $H_{s+1}^{r^k}(n)$ has $r^k$ indexes for each $n=1,\dots, r-1\longrightarrow$ total $=1+(r-1)+r(r-1)+\cdots + r^k(r-1)=1+(r-1)(\frac{r^{k+1}-1}{r-1})=r^{k+1}$. \hspace*{1.1cm}
\end{itemize}
Thus, if $\lambda_l\in H_{s+1}^{r^k}(n)$ for some $n=1,\dots,r-1$, its enumeration cannot go past the last index inside $H_{s+1}^{r^k}(r-1)$, that is, \[l\leq r^{k+1}.\] If we want to be even more precise, supposing that $\lambda_l\in H_{s+1}^{r^{k_l}}(r_l)$, one can check that \[r_lr^{k_l}\leq l \leq (r_l+1)r^{k_l},\] although we will not need this much precision. Another practical aspect of these subdivision parts with rank and fineness is that, if we know that an index $j$ belongs to some subdivision that appears \emph{before} $H_{s+t_0}^{r^{k_0}}(n_0)$ and that $l$ belongs to another subdivision that appears \emph{after} $H_{s+t_0}^{r^{k_0}}(n_0)$, then all the indexes inside $H_{s+t_0}^{r^{k_0}}(n_0)$ must have been counted in the enumeration from $\Gamma_j$ to $\Gamma_l$, we then immediately get the estimate $l-j\geq r^{k_0}$.

The next step in our construction consists of proving that the tagged covering $\big(\lambda_k,\Gamma_k\big)_{k=1,\dots,r^t}$ satisfies (\ref{cerise}). So let $1\leq j<l\leq r^t$ and $\lambda\in\Gamma_j, \mu\in\Gamma_l$. Let $r_j, r_l, k_j, k_l$ be such that $\lambda\in H_{s+1}^{r^{k_j}}(r_j)$ and $\mu\in H_{s+1}^{r^{k_l}}(r_l)$. The proof can be broken into two situations.
\begin{enumerate}[(a)]
    \item $\|\lambda_l-\lambda_j\|<c^{s+1}\rho$,
    \item $\|\lambda_l-\lambda_j\| \geq c^{s+1}\rho$.
\end{enumerate}
Situation (b) is the easiest, so let us begin by that one. Suppose that $\|\lambda_l-\lambda_j\| \geq c^{s+1}\rho$. Since this distance is greater than any rank $s+1$ subdivision and $l$ belongs to a subdivision with fineness $k_l$, we conclude that the enumeration from $j$ to $l$ counts at least one rank $s+2$ subdivision with fineness at least $r^{k_l-2}$ (which can happens in the case $r_l=1$, see Figure \ref{fully-counted-l-1} for an example).
\begin{figure}[H]
    \centering
    \includegraphics[width=0.50\textwidth]{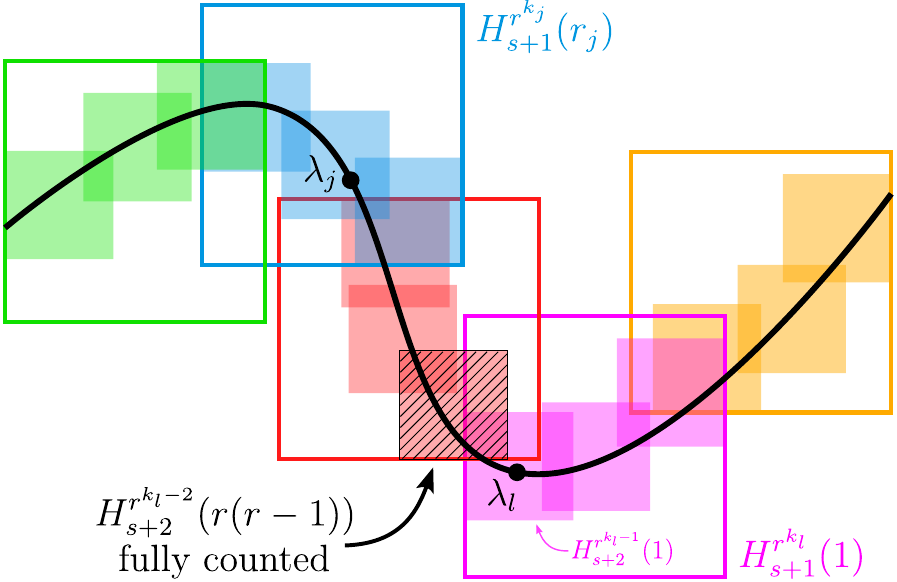}
    \caption{Portion of $\Gamma$. One rank $s+2$ subdivision with fineness at least $k_l-2$ is fully counted.}
    \label{fully-counted-l-1}
\end{figure}
We use the bound $l\leq r^{k_l+1}$ in order to get \[l-j\geq r^{k_l-2}\implies \frac{l-j}{l}\geq\frac{r^{k_l-2}}{r^{k_l+1}}\geq \frac{1}{r^3}.\]
So in this case we easily get \[\|\lambda-\mu\|\leq \rho=c^{-3}\rho c^3=c^{-3}\rho\bigg(\frac{1}{r^3}\bigg)^\alpha\leq \frac{\rho}{c^3}\bigg(\frac{l-j}{l}\bigg)^\alpha\leq D\bigg(\frac{l-j}{l}\bigg)^{1/\gamma}.\]

We now consider situation (a), that is, we suppose $\|\lambda_l-\lambda_j\|<c^{s+1}\rho$. We have two cases to consider.\vspace{3mm}

\noindent\textbf{1st case.} The rank $s+1$ subdivision parts to which $\lambda_l$ and $\lambda_j$ belong are at least one full rank $s+1$ subdivision apart. In other words: either $k_l=k_j+1$ and $r_j=r_l$; or $k_l>k_j+1$ (see Figure \ref{fullycounted-l.pdf} for an example).\vspace{3mm}

\begin{figure}[H]
    \centering
    \includegraphics[width=0.55\textwidth]{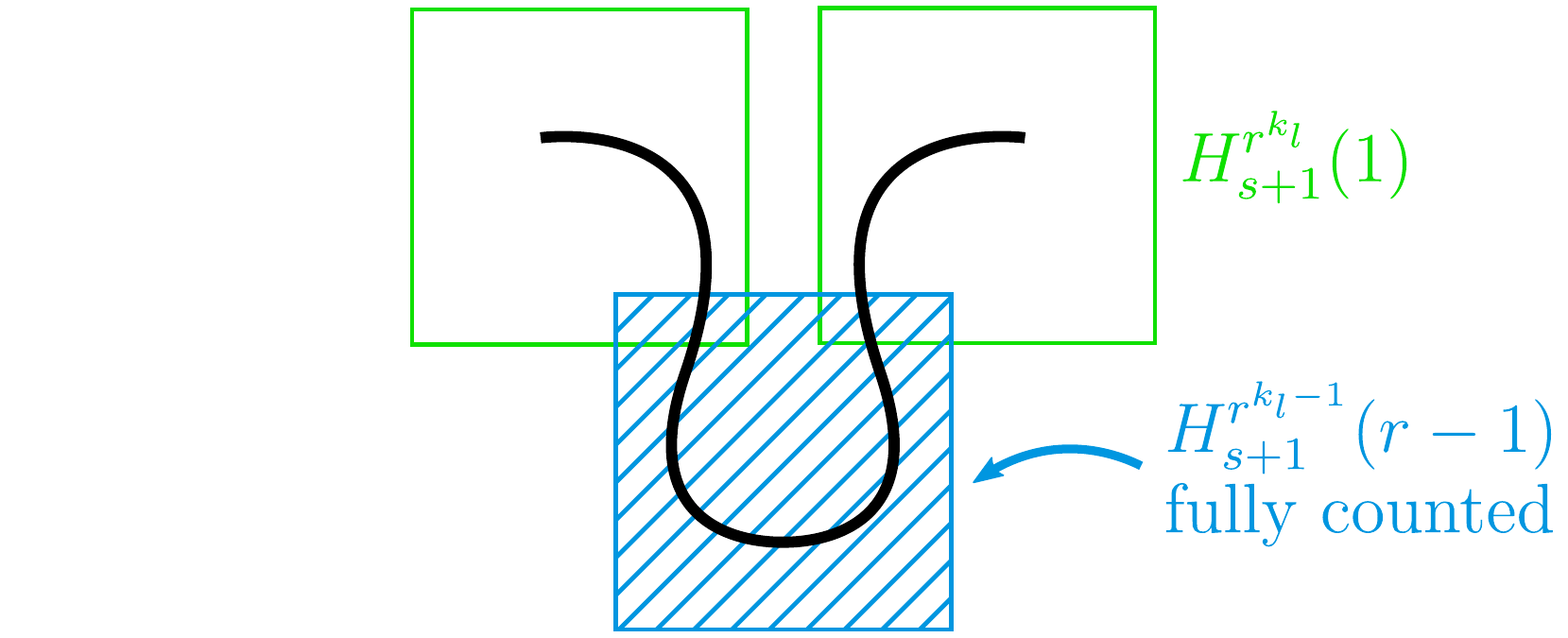}
    \caption{Portion of $\Gamma$. One full rank $s+1$ subdivision with fineness at least $r^{k_l-1}$ is fully counted.}
    \label{fullycounted-l.pdf}
\end{figure}

Then, similarly to case (b), the enumeration from $j$ to $l$ counts a full rank $s+1$ subdivision with fineness at least $k_l-1$, what implies
\[
    l-j\geq r^{k_l-1}\implies\frac{l-j}{l}\geq\frac{r^{k_l-1}}{r^{k_l+1}}\geq \frac{1}{r^2},
\]
hence,
\[\|\lambda-\mu\|\leq \rho=c^{-2}\rho c^2=c^{-2}\rho\bigg(\frac{1}{r^2}\bigg)^\alpha
\leq 
\frac{\rho}{c^3}\bigg(\frac{l-j}{l}\bigg)^\alpha\leq D\bigg(\frac{l-j}{l}\bigg)^{1/\gamma}.\]

\noindent\textbf{2nd case.} The rank $s+1$ subdivision parts to which $\lambda_l$ and $\lambda_j$ belong are either the same or are neighbors in the enumeration. In other words: either $k_l=k_j$ and $r_j=r_l$; or $k_l=k_j$ and $r_l=r_j+1$; or $k_l=k_j+1$ and $r_l=1, r_j=r-1$.
\vspace{3mm}

For the sake of easy counting, we can look at $H_{s+1}^{k_l}(\gamma_l)$ as having the coarser fineness $k_j$ (see Figure \ref{simplification} for an example of simplification). 
\begin{figure}[H]
    \centering
    \includegraphics[width=0.95\textwidth]{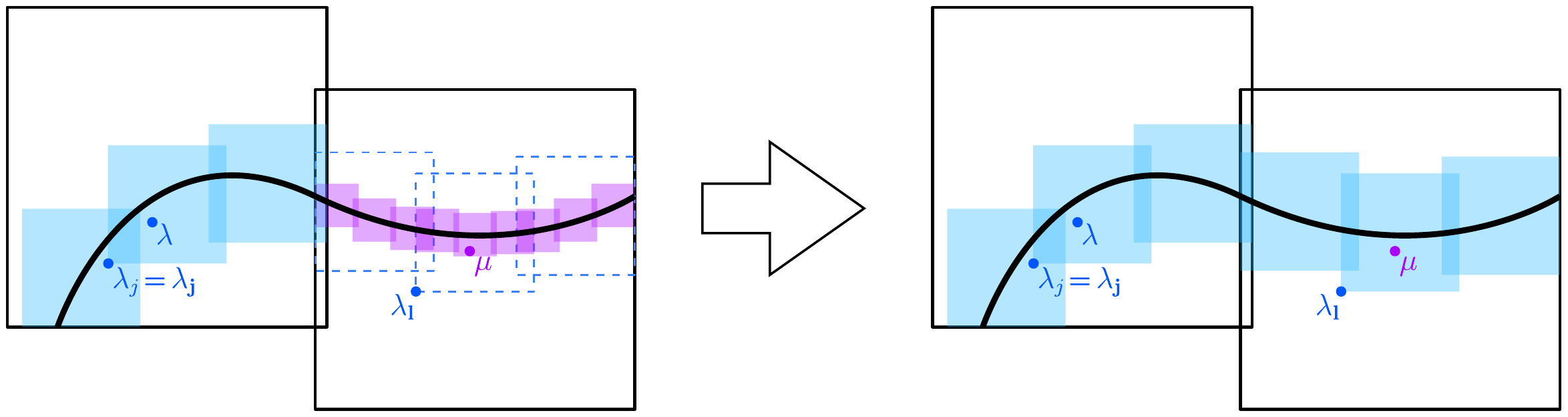}
    \caption{Simplification of a resolution transition. We reduce both to the slowest resolution for easy counting. Some indexes are lost.}
    \label{simplification}
\end{figure}
This is only needed if $k_l=k_j+1$, in which case we lose some indexes, but nothing that will ruin our proof. If we denote by ${\bf j}$ the resolution $k:=s+1+k_j$ tag that coincides with $\lambda_j=\lambda_{\bf j}$ and by ${\bf l}$ the resolution $k$ tag which corresponds to the resolution $k$ part $\Lambda_{\bf l}$ containing $\mu$, we can identify $\lambda_l$ to $\lambda_{\bf l}$ and notice that $l-j\geq {\bf j}\to{\bf l}$. Let $n$ be such that \[\frac{c^k\rho}{c^n}\leq \|\lambda_{\bf l}-\lambda_{\bf j}\|\leq \frac{c^k\rho}{c^{n+1}}.\]
We apply Lemma \ref{dSlemma:general} and get 
\begin{align*}
l-j\geq {\bf j}\to {\bf l}\geq \frac{r^{n-1}+r-2}{r-1}
&\implies \frac{l-j}{l}\geq \frac{\frac{r^{n-1}+r-2}{r-1}}{r^{k_l+1}}\geq \frac{1}{r-1}\times\frac{r^{n-1}}{r^{k_j+1+1}}\\
&\implies \bigg(\frac{l-j}{l}\bigg)^{\alpha}\geq \frac{1}{(r-1)^\alpha}c^{k_j-n+3}.
\end{align*}
We then combine this lower bound with the estimate
\begin{align*}
    \|\lambda-\mu\|
        &\leq\|\lambda-\lambda_{\bf j}\|+\|\lambda_{\bf l}-\lambda_{\bf j}\|+\|\lambda_{\bf l}-\mu\| \\
        &\leq c^{k_j+s+1}\rho+\frac{c^{k}\rho}{c^{n+1}}+c^{k_l+s+1}\rho\\
        &\leq 3\rho c^{k-n-1}
\end{align*}
in order to obtain
\[
\frac{\big(\frac{l-j}{l}\big)^\alpha}{\|\lambda-\mu\|}
\geq \frac{\frac{1}{(r-1)^\alpha}c^{k_j-n+3}}{3\rho c^{k-n-1}}=\frac{1}{3\rho(r-1)^\alpha}\times\frac{c^{k_j-n+3}}{c^{s+1+k_j-n-1}}= \frac{c^3}{3\rho(r-1)^\alpha c^{s}}.
\]
Remember that we have assumed that $3(r-1)^\alpha c^s\leq 1$ (which is true if $s$ is big enough). Hence,
\[\|\lambda-\mu\|\leq \frac{\rho}{c^3} 3(r-1)^\alpha c^{s}\bigg(\frac{l-j}{l}\bigg)^\alpha\leq \frac{\rho}{c^3}\bigg(\frac{l-j}{l}\bigg)^\alpha.\]
This completes the proof.
\section{Open problems and perspectives} \label{sec:prob}

Through the combination of the concept of HBD from \cite{BCQCommon} and the idea of using space-filling curves for ordering coverings, we have achieved optimal applications in a large class of self-similar fractals. However, numerous questions remain open. We will state some of them here, although many others can be made by looking at our construction from different angles. Specialists on fractal geometry, in special, are more likely to expand some of our constructions to fit larger classes of fractals.

The main idea that have allowed us to improve previous results and attain optimally is the use of space-filling curves to avoid jumps in the enumeration. Hence, the obvious question that can be asked is: what can we do when jumps cannot be avoided? This is the case of totally disconnected (dust-like) fractals, like the Cantor dust for example. A ``break of nature'' happens with these fractals: we go from having no jumps to having jumps on every single step. The simplest version of this problem can be stated as follows.
\begin{problem}
    Let $\C\subset \RR_+^2$ be a Cantor dust with uniform dissection ratio $1/4$. Does the family of products of Rolewicz operators $\big(e^\lambda B\times e^\mu B\big)_{(\lambda,\mu)\in\C}$ acting on $c_0\times c_0$ have a common hypercyclic vector?
\end{problem}
Notice that \cite[Example 4.9]{BCQCommon} in combination with Borichev's example give the answer for all Cantor dusts with uniform dissection ratio $c\neq \frac{1}{4}$. The limit case $c=\frac{1}{4}$ remains open.

When it comes to self-similar connected fractals, one can notice that the mere existence of a homogeneously ordered family of partitions provides a family of curves converging pointwise to the fractal itself. In fact, if $\big((\Lambda_{{\bf i}})_{{\bf i}\in I_r^m}\big)_{m\in \NN}$ is a homogeneously ordered family of coverings for $\Lambda$, then for each $m\in\NN$, we fix one point $p_{\bf i}$ inside each part $\Lambda_{\bf i}$, ${\bf i}\in I_r^m$, and we link them all following the lexicographical ordering of $I_r^{m}$. This defines a polygonal path $P_m=[p_{1,...,1}, ..., p_{r,...,r}]$, for all $m\in\NN$. Then $\Lambda=\lim_{m\to+\infty}P_m$ where the convergence is pointwise. However, although homogeneity and ordering of the family might imply the continuity of the limit curve, calculating its Hölder exponent is a more delicate task. It is, then, much more natural to apply our result by taking into account the number of similarities and their contraction ratio rather than looking at it as a Hölder curve.

Particularly for self-similar fractals that are not uniformly contracting, we can still give applications. In fact, if the fractal is homogeneously ordered, then we can consider the biggest contraction ratio everywhere in the coverings and, in this sense, treat the fractal as being uniformly contracting. More precisely, if $\Lambda$ is a homogeneously ordered fractal with $r$ similarities and contraction ratios $c_1, \cdots c_r$, then it has HBD$^\circ$ at most
\begin{equation} \label{eq:max:sim}
    \gamma = \max\bigg\{-\frac{\log r}{\log c_i}: i=1,\dots,r\bigg\}.
\end{equation}
The result is not optimal of course. The following question remains open.
\begin{problem}\label{prob:general-dim}
    Can we refine the homogeneously ordered box dimension and adapt our construction in order to include non-uniformly contracting self-similar sets? Can this be done optimally (that is, by using their similarity dimension)?
\end{problem}

Although the ``get around'' method \eqref{eq:max:sim} for calculating the HBD$^\circ$ of non-uniformly contracting self-similar can provide many more examples, sometimes it has to be used carefully, otherwise the non-optimal result is far from being interesting. Take for example the generalized Von Koch curve (defined in \cite[Page 64]{mandelbrot} and further studied in \cite{dekking}), which has 4 similarities with respective contraction ratios $1/4$, $1/2$, $1/2$ and $1/4$. Although its similarity dimension is $\frac{\log(1+\sqrt{3})}{\log(2)}<1.5$, a direct application of \eqref{eq:max:sim} implies that the generalized Von Koch curve has HBD$^\circ$ at most $2$ (which is the worse estimate one could obtain).
\begin{figure}
    \centering
    \includegraphics[width=0.95\textwidth]{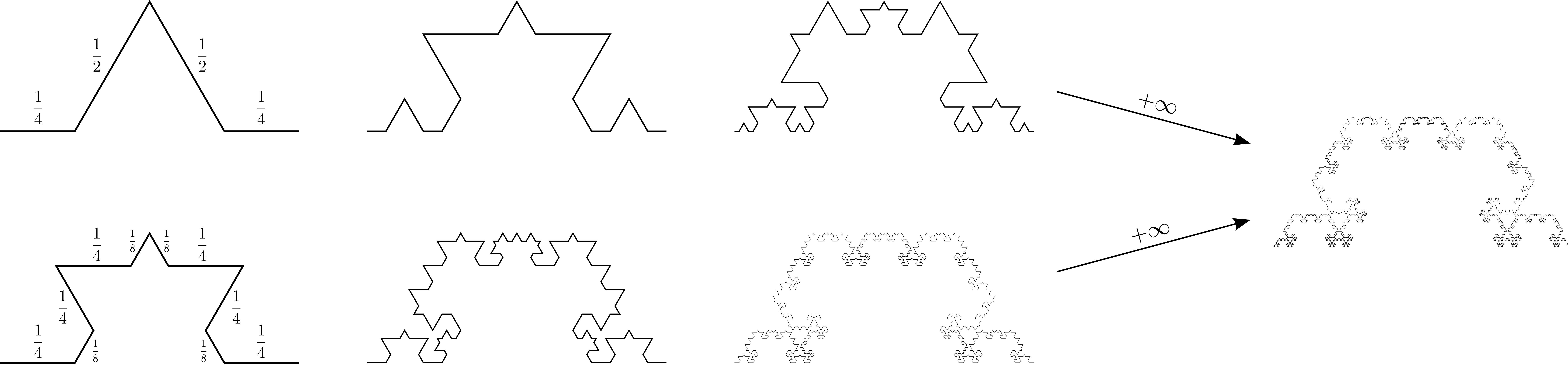}
    \caption{Classical and alternative construction of the generalized Von Koch curve}
    \label{fig:koch-non-uni}
\end{figure}
Following a suggestion by A. Peris, the approach \eqref{eq:max:sim} can be better applied if we increase the number of similarities of the non-uniformly contracting fractal. In the case of the generalized Von Koch curve, for example, we can change the seed by increasing the resolution in the second and third similarities, what makes the distribution of its contractions more uniform. By doing this, the final curve obtained through the iterative process is the same, but it is now interpreted as having 10 similarities with respective contraction ratios $1/4$, $1/8$, $1/4$, $1/4$, $1/8$, $1/8$, $1/4$, $1/4$, $1/8$, $1/4$, as represented in Figure \ref{fig:koch-non-uni}. Applying \eqref{eq:max:sim} in this alternative form, we get that the generalized Von Koch curve has HBD$^\circ$ at most $\gamma=\frac{\log 10}{\log 4}\approx 1.66$. By increasing even more the number of similarities, it is possible to prove that the generalized Von Koch curve has HBD$^\circ$ at most $\gamma$ for any gamma $\gamma>\frac{\log(1+\sqrt{3})}{\log(2)}$. It is not clear if we can get an optimal result in this example. The trick of increasing the number of similarities can be more generally applied to any self-similar fractal. When the fractal is uniformly contracting, Theorem \ref{cs-gamma-dim} provides an optimal application. If it is non-uniformly contracting, then Theorem \ref{cs-gamma-dim} applies almost optimally, that is, we can apply \eqref{eq:max:sim} and obtain an application for all $\gamma$ strictly bigger than the similarity dimension of the fractal. This problem comes from the fact that, in the statement of our result, we do not use the HBD$^\circ$ itself, but rather the fact that the set of parameters has HBD$^\circ$ \emph{at most} some positive value. The optimal application is obtained when the fractal has HBD$^\circ$ at most its similarity dimension, which is the case when it is self-similar and uniformly contracting.
\begin{problem}\label{prob:hbd}
    Instead of assuming that the parameters has HBD$^\circ$ at most some value, is it possible to use the HBD$^\circ$ itself for constructing a covering as in Lemma \ref{key:lemma}?
\end{problem}

We can look at the same problem from the optimal parametrization point of view. In \cite[Theorem 7.1]{hata}, the author shows that, in some cases, the invariant set $\Lambda$ of an IFS with non-uniform contraction ratio can be parametrized by an $\beta$-Hölder function, where $\beta=1/\gamma$ and $\Lambda$ happens to have homogeneous box dimension at most $\gamma$. This should be the case, for instance, of the generalized Von Koch curve, but the Hölder exponent is obtained though the same maximum as in \eqref{eq:max:sim}. Thus, the optimality problem for parametrizations is the same as for the HBD$^\circ$ approach. We can ask for example the following.
\begin{problem}
    Is it possible to optimally parametrize the generalized Von Koch curve?
\end{problem}
If the answer is positive, than our results Theorem \ref{cs-gamma-dim} and Corollary \ref{corol:intro} apply optimally. If not, then this would further justify the search for a positive answer to Problems \ref{prob:general-dim} or \ref{prob:hbd}, for the optimal parametrization would fail in this case. There are known cases where this happens. Consider, for example, the graph of the Takagi function (see \cite{takagi} for more details). It is a curve that has Hausdorff dimension 1, that is $\beta$-Hölder for all $\beta\in (0,1)$ but is not (even locally) Lipschitz. Hence, it is impossible to reach optimality through a parametrization approach. At the same time, it seems unlikely that this curve has HBD$^\circ$ at most 1 (although it might have HBD$^\circ$ at most $\gamma$ for all $\gamma>1$). Still, it is not impossible that a better covering approach would lead to a positive result to the following question.
\begin{problem}\label{prob:takagi}
    Let $\Gamma\subset \RR^2_+$ be the graph of a Takagi function. Is it true that $\bigcap_{(x,y)\in \Gamma}HC(e^xB\times e^y B)$ has a common hypercyclic vector on $c_0\times c_0$?
\end{problem}

The following problem is a fundamental question on common hypercyclicity in higher dimensions. It has motivated the study of the family of weights that we have considered in this paper.
\begin{problem}\label{prob:general}
    Let $I\subset \RR_+$ be a compact interval and consider the family of weights $\big(w(x)\big)_{x\in I}$ such that $w_1(x)\cdots w_n(x)\approx \exp(xF(n))$, $x\in I$, $n\in\NN$, for some function $F:\NN\to\RR_+$. Does $\sum \frac{1}{F(n)^2}=+\infty$ imply that the family $\big(B_{w(x)}\times B_{w(y)}\big)_{(x,y)\in I^2}$ has a common hypercyclic vector on $c_0\times c_0$? What conditions on $F$ are needed?
\end{problem}
In this article, we have worked with functions of the form $F(n)=C_0n^\alpha$, for some $C_0>0$ and $\alpha\in(0,1]$. Problem \ref{prob:general} states that the study of more general $F$ is a natural continuation. A more specific motivational exemple is the family $\mathcal{F}:=(T_\alpha)_{\alpha\in I}$ acting on a Fréchet sequence space $X$, where $I\subset [0,1]$ and, for every $\alpha\in I$, $T_\alpha$ is the weighted backward shift induced by $w_1(\alpha)\cdots w_n(\alpha)=\exp(n^\alpha)$. Then $\mathcal{F}$ doesn't satisfy the one dimensional CS-Criterion as it is currently stated, but it is not difficult to prove that this family is common hypercyclic by constructing a partition of $[0,1]$ by sub-intervals of size $\frac{1}{n_i\log(n_i)}$, where $n_i=i N$ for all $i,N\in\NN$. In this sense, a different version of the CS-Criterion could be stated by substituting $\frac{\tau}{n}$ for $\frac{\tau}{n\log(n)}$ in its statement. There are multiple ways of exploring these ideas in higher dimensions.

\section*{Aknowledgments}
The author would like to thank Professor Julien Cassaigne for a question that lead us to the study of Hilbert curves, Professor Stéphane Charpentier for many insightful discussions, the Mathematics Laboratory of Avignon for the financial support and the anonymous referee for their pertinent suggestions.

\bibliographystyle{elsarticle-harv}
\bibliography{bibliography}

\end{document}